\documentclass[a4paper,fleqn]{cas-sc}

\usepackage[authoryear,longnamesfirst]{natbib}

\usepackage{amssymb}
\usepackage{amsmath}

\usepackage{bm}
\usepackage{graphicx}
\usepackage{fancyhdr}
\usepackage{epstopdf}
\usepackage{amsfonts}
\usepackage{amsthm}
\usepackage{mathtools}
\usepackage{enumitem}
\usepackage[percent]{overpic}
\usepackage{tikz}
\usepackage{xargs}
\usepackage{xcolor}

\usepackage{comment}
\usepackage{hyperref}
\usepackage[normalem]{ulem}
\usepackage[caption=false]{subfig}
\newcommand{\R}{\mathbb{R}}
\newcommand{\C}{\mathbb{C}}

\newcommand{\Z}{\mathbb{Z}}
\newcommand{\bigO}{\mathcal{O}}
\newcommand{\RB}{BaMRAM}
\newcommand{\RS}{SpaMRAM}
\newcommand{\norm}[1]{\left\lVert#1\right\rVert}
\newcommand{\vct}[1]{\bm{#1}}
\newcommand{\mtx}[1]{\bm{#1}}
\newcommand{\flr}[1]{\left\lfloor {#1} \right\rfloor}
\newcommand{\cei}[1]{\left\lceil{#1}\right\rceil}
\DeclareMathOperator*{\argmin}{arg\,min}
\usepackage{algorithm, algorithmicx, algpseudocode}

\newcommand{\ignore}[1]{}

\newtheorem{theorem}{Theorem}[]

\newtheorem{proposition}[theorem]{Proposition}

\newtheorem{remark}[theorem]{Remark}

\def\tsc#1{\csdef{#1}{\textsc{\lowercase{#1}}\xspace}}
\tsc{WGM}
\tsc{QE}
\tsc{EP}
\tsc{PMS}
\tsc{BEC}
\tsc{DE}

\begin{document}
\let\WriteBookmarks\relax
\def\floatpagepagefraction{1}
\def\textpagefraction{.001}
\shorttitle{Approximating Sparse Matrices and their Functions using Matrix-vector products}
\shortauthors{T. Park and Y. Nakatsukasa}

\title [mode = title]{Approximating Sparse Matrices and their Functions using Matrix-vector products}                      



\author[1]{Taejun Park}[type=editor
                        ]
\cormark[1]
\fnmark[1]
\ead{taejun.park@epfl.ch}


\affiliation[1]{organization={Institute of Mathematics},
                addressline={EPF Lausanne}, 
                city={1015 Lausanne},
                country={Switzerland}}

\author[2]{Yuji Nakatsukasa}[type=editor]
\fnmark[2]

\ead{nakatsukasa@maths.ox.ac.uk}


\affiliation[2]{organization={Mathematical Institute},
                addressline={University of Oxford}, 
                city={Oxford},
                postcode={OX2 6GG},
                country={UK}}


\cortext[cor1]{Corresponding author}
\fntext[fn1]{This work was done partially while the author was at the University of Oxford. TP is supported by the RandESC project funded by the Swiss Platform for Advanced Scientific Computing. TP was also supported by the Heilbronn Institute for Mathematical Research.}
\fntext[fn2]{YN is supported by EPSRC grants EP/Y010086/1 and EP/Y030990/1.}


\begin{abstract}
The computation of a matrix function $f(A)$ is an important task in scientific computing appearing in machine learning, network analysis and the solution of partial differential equations. In this work, we use only matrix-vector products $x\mapsto Ax$
to approximate functions of sparse matrices and matrices with similar structures such as sparse matrices $A$ themselves or matrices that have a similar decay property as matrix functions. We show that when $A$ is a sparse matrix with an unknown sparsity pattern, techniques from compressed sensing can be used under natural assumptions. Moreover, if $A$ is a banded matrix then certain deterministic matrix-vector products can efficiently recover the large entries of $f(A)$. We describe an algorithm for each of the two cases and give error analysis based on the decay bound for the entries of $f(A)$. We finish with numerical experiments showing the accuracy of our algorithms.
\end{abstract}




\maketitle

\section{Introduction} \label{sec:intro}
The evaluation of matrix functions is an important task arising frequently in many areas of scientific computing. They appear in the numerical solution of partial differential equations \cite{pde1,pde2,pde3}, network analysis \cite{netanal1,EstradaHigham2010}, electronic structure calculations \cite{elec2,elec1} and statistical learning \cite{CortinovisKressner2021,Wengeretal2022}, among others. In some problems, the matrix of interest has some underlying structure such as bandedness or sparsity. For example, in graphs with community structure one requires functions of sparse matrices \cite{Newman2006} and for electronic structure calculations one computes functions of banded matrices \cite{elec2,elec1}. In this paper, we consider the problem of approximating $f(\mtx{A})$ through matrix-vector products $\mtx{x}\mapsto\mtx{Ax}$ only, assuming $\mtx{A}$ is sparse or banded. 

Let $\mtx{A}\in \R^{n\times n}$ and $f$ be a function that is analytic on the closure of a domain $\Omega \subset \C$ that contains the eigenvalues of $\mtx{A}$. Then the matrix function $f(\mtx{A})$ is defined as
\begin{equation*}
    f(\mtx{A}) = \frac{1}{2\pi i} \int_{\partial \Omega} f(z)(z\mtx{I}-\mtx{A})^{-1} dz,
\end{equation*} where $\mtx{I}$ denotes the identity matrix and $(z\mtx{I}-\mtx{A})^{-1}$ is called the resolvent matrix of $\mtx{A}$. If $\mtx{A}\in \R^{n\times n}$ is diagonalizable with $\mtx{A} = \mtx{X}\mtx{D}\mtx{X}^{-1}$ where $\mtx{X}$ is the eigenvector matrix and $\mtx{D}$ is the diagonal matrix containing the eigenvalues $\{\lambda_i\}_{i=1}^n$ of $\mtx{A}$, then
\begin{equation*}
    f(\mtx{A}) = \mtx{X}f(\mtx{D})\mtx{X}^{-1}, \hspace{1em} f(\mtx{D}) =  \begin{bmatrix}
    f(\lambda_1) & & \\
    & \ddots & \\
    & & f(\lambda_n)
    \end{bmatrix}.
\end{equation*}
There are many other equivalent ways of defining $f(\mtx{A})$, see for example \cite{matfunchigham}. 

When the matrix dimension is large, explicitly computing $f(\mtx{A})$ using standard algorithms \cite{matcomp,matfunchigham} becomes prohibitive in practice, usually requiring $O(n^3)$ operations. Moreover, the matrix $\mtx{A}$ may also be too large to store explicitly. In this scenario, $\mtx{A}$ is often only available through matrix-vector products and $f(\mtx{A})\vct{b}$ \footnote{We note that the matrix-vector product $f(\mtx{A})\vct{b}$ is not computed by explicitly forming $f(\mtx{A})$ and then multiplying it by $\vct{b}$. Instead, it is approximated using matrix-vector products involving $\mtx{A}$ alone, typically through Krylov subspace methods.} is used to infer information about $f(\mtx{A})$ \cite{Guttel2013,GuttelKressnerLund2020}. In this paper, we show that the whole matrix $f(\mtx{A})$ can still be approximated in a sparse format using $\ll n$ matrix-vector products with $\mtx{A}$ if $\mtx{A}$ is sparse and $f$ is smooth. We exploit the decay bounds for the entries of $f(\mtx{A})$ \cite{Benzi2007} to recover its large entries using matrix-vector products; see also \cite{Benzi2016}. We primarily study functions of sparse matrices with general sparsity pattern and matrices with few large entries in each row/column (see Section \ref{sec:sparse}), but also consider functions of banded matrices and matrices with off-diagonal decay as a special case (see Section \ref{sec:band}). Our work is particularly useful if $f(\mtx{A})$ is sparse or has only a few entries that are large in each row/column or if $\mtx{A}$ is banded with bandwidth $k \ll n$. Throughout this work, we assume that the function $f$ is analytic on the closure of a domain that contains the eigenvalues of $\mtx A$.

Let $\mtx{A}$ be a diagonalizable, sparse matrix with some sparsity pattern and $f$ (either $f:\mathbb{R}\rightarrow \mathbb{R}$ or $f:\mathbb{C}\rightarrow \mathbb{C}$) a function that is analytic in a region containing the spectrum of $\mtx{A}$. Depending on the sparsity pattern of $\mtx{A}$ and the smoothness of $f$, some entries of $f(\mtx{A})$ can be shown to be exponentially smaller than others \cite{Benzi2007} and we can view them as noise. For a sparse matrix with a general sparsity pattern, it can be difficult to know the support (location of the large entries) of $f(\mtx{A})$ in advance, for example, if the sparsity pattern of $\mtx{A}$ is unknown. However, we can write $f(\mtx{A}) = \widehat{f(\mtx{A})}+ \mtx{E}$ with the matrix $\mtx{E}$ having only small entries in magnitude and compute $\widehat{f(\mtx{A})}$ by viewing $\widehat{f(\mtx{A})}$ as signals and $\mtx{E}$ as noise using \emph{compressed sensing} row by row or column by column. Compressed sensing is a technique in signal processing for efficiently reconstructing a signal which is sparse in some domain by finding solutions to underdetermined linear systems \cite{CandesRombergTao06,CandesTao06, Donoho06}. As a special case (by taking $f(\mtx{A})=\mtx{A}$), our algorithm can be used for recovering a sparse matrix $\mtx{A}$ itself, where $\mtx{A}$ is accessed only through matrix-vector multiplications $\vct{x}\mapsto \mtx{A}\vct{x}$. The details will be discussed in Section \ref{sec:sparse}.

Let $\mtx{A}$ now be a banded matrix. Then $f(\mtx{A})$ can usually be well approximated by a banded matrix \cite{DMS1984}. There are theoretical results that confirm this property \cite{BenziGolub}, which show that the entries of $f(\mtx{A})$ decay exponentially away from the main diagonal. The observation is that $\left|[f(\mtx{A})]_{ij}\right| \ll \left|[f(\mtx{A})]_{ii}\right|,\left|[f(\mtx{A})]_{jj}\right|$ for $|i-j|\gg 1$, that is, the large entries are typically located not too far away from the main diagonal. Therefore, we can dismiss the exponentially small entries and find a matrix that recovers the large entries of $f(\mtx{A})$, i.e., find a matrix $\widehat{f(\mtx{A})}$ such that the difference $f(\mtx{A})-\widehat{f(\mtx{A})}$ only contains small entries in magnitude. We can show that the large entries of $f(\mtx{A})$ can be captured through matrix-vector products with the following simple deterministic matrix
\begin{equation*}
    \mtx{I}_{n}^{(s)}:= \left[\mtx{I}_s,\mtx{I}_s,...,\mtx{I}_s,\mtx{I}_s(:,1:\{n/s\}\cdot s)\right]^T \in \R^{n\times s}
\end{equation*} where $s<n$, $\mtx{I}_s$ is the $s\times s$ identity matrix and $\mtx{I}_s(:,1:\{n/s\}\cdot s)$ are the first $\{n/s\}\cdot s$ columns of $\mtx{I}_s$, which ensures that $\mtx{I}_{n}^{(s)}$ has exactly $n$ rows. Here $\{x\}$ is the fractional part of $x$. This is unlike the sparse case where we now have significant information on the support of $f(\mtx{A})$. This set of vectors $\mtx{I}_{n}^{(s)}$ are identical to the probing vectors proposed in \cite{FrommerSchnimmelSchweitzer2021} for estimating functions of banded matrices and \cite{ColemanMore1983,CPR1974} for estimating sparse, banded Jacobian matrices. If we let $s$ be proportional to the bandwidth of $\mtx{A}$, the matrix-vector products $f(\mtx{A})\mtx{I}_n^{(s)}$ capture the large entries near the main diagonal, giving us a good approximation for $f(\mtx{A})$. We discuss the details in Section \ref{sec:band}.

Both cases require matrix-vector products with $f(\mtx{A})$, i.e., $\vct{x}\mapsto f(\mtx{A})\vct{x}$. Computing this matrix-vector product to high accuracy is crucial for our algorithms. Popular methods include the standard (polynomial) Krylov method \cite{matfunchigham,Saad92}, and contour integration \cite{matveccontour}. In this paper, we discuss three functions: the exponential $e^{\mtx{A}}$, the square root function $\sqrt{\mtx{A}}$ and the logarithm $\log(\mtx{A})$. For the exponential we use the polynomial Krylov method, and for the square root function and the logarithm, we use contour integration as in {\cite[method 2]{matveccontour}} as both functions have a branch cut on the negative real axis $(-\infty,0]$, making them difficult to approximate using polynomials if $\mtx{A}$ has eigenvalues close to $0$. For Krylov methods, block versions can also be used \cite{FrommerLundSzyld2018}, which compute $f(\mtx{A})\mtx{B} \in \R^{n\times s}$ at once instead of column by column. Although the computation of the matrix-vector product $f(\mtx{A})\vct{b}$ is important and crucial to the approximation of $f(\mtx{A})$, we do not discuss the details further as they are not the focus of this paper; see for example \cite[Ch.13]{matfunchigham}.

\subsection{Existing methods}
First, Benzi and Razouk \cite{Benzi2007} describe an algorithm that exploits the decay property by first finding a good polynomial approximation $p$ to $f$, for example using Chebyshev interpolation, and computing $p(\mtx{A})$ while using a dropping strategy to keep the matrix as sparse as possible. However, this method can quickly become infeasible if $f$ cannot be well approximated by a low-degree polynomial or if sparse data structures and arithmetic is not available. Frommer, Schimmel and Schweitzer \cite{FrommerSchnimmelSchweitzer2021} use probing methods to approximate matrix functions $f(\mtx{A})$ and its related quantities, for instance the trace of $f(\mtx{A})$. The method finds probing vectors to estimate the entries of $f(\mtx{A})$ by partitioning the vertices of the directed graph $G(\mtx{A})$ associated with the sparse matrix $\mtx{A}$ using graph coloring and the sparsity pattern of $\mtx{A}$. Our Section~\ref{sec:band} on banded matrices can be seen as a special case of~\cite{FrommerSchnimmelSchweitzer2021}, but as banded matrices arise naturally in applications and come with strong theoretical results, we make the sensing matrix explicit and study them in detail. The strategy proposed in \cite{FrommerSchnimmelSchweitzer2021} costs $\bigO(\Delta^d n)$ to obtain the graph coloring where $\Delta$ is the maximal degree of $G(\mtx{A})$ and $d$ is a distance parameter. In the special case when $\mtx{A}$ is a banded matrix the graph coloring can be obtained with $\bigO(n)$ cost. Our work and \cite{FrommerSchnimmelSchweitzer2021} satisfies a similar theoretical accuracy bound up to a low-degree polynomial factor in $n$ (cf. Section 3 of \cite{FrommerSchnimmelSchweitzer2021} and Theorem \ref{thm:gen} in our work), but we use compressed sensing to estimate the large entries of each row/column.
Cortinovis, Kressner and Massei \cite{CortinovisKressnerMassei2022} use a divide-and-conquer approach to approximate matrix functions where the matrix is either banded, hierarchically semiseparable or has a related structure. The divide-and-conquer method is based on (rational) Krylov subspace methods for performing low-rank updates of matrix functions. Under the assumption that the minimal block size in the divide-and-conquer algorithm is $\bigO(1)$ and the low-rank update in the Krylov method converges in $\bigO(1)$ iterations, the algorithm runs with complexity $\bigO(nk^2)$ for $k$-banded matrices.

There are other related works that use matrix-vector products to recover matrices with certain properties or using compressed sensing to recover matrices with only few large entries. In \cite{Townsend22,LetvittMartinsson2022,Woodruffetal2021}, the authors devise algorithms or strategies based on matrix-vector products for recovering structured matrices such as hierarchical low-rank matrices and Toeplitz matrices or for answering queries such as whether an unknown matrix is symmetric or not. In particular, Curtis, Powell and Reid \cite{CPR1974} and later Coleman and Mor\'e \cite{ColemanMore1983} use a certain graph coloring of the adjacency graph of a sparse or banded matrix to approximate sparse or banded Jacobian matrices using matrix-vector products. Bekas, Kokiopoulou and Saad \cite{BekasKokiopoulouSaad2007} suggested using the same set of coloring as the banded case to recover matrices with rapid off-diagonal decay.

In \cite{HermanStrohmer2009,PfranderRauhutTanner2008}, the authors formulate an algorithm based on compressed sensing techniques to recover sparse matrices with at most $\bigO(sn/\log^2 n)$ nonzero entries using $s$ matrix-vector products (See {\cite[Thm. 6.3]{PfranderRauhutTanner2008}}). This method uses a vectorized form of the matrix in question after a certain transformation, which can become expensive. Lastly, Dasarathy, Shah and Bhaskar \cite{Dasarathyetal2015} use compressed sensing techniques by applying compression (sensing operators) on both sides, i.e. $\mtx{Y}_1 \mtx{X}\mtx{Y}_2^T$, to recover an unknown sparse matrix $\mtx{X}$. The authors use $\bigO(\sqrt{n}\log n)$ matrix-vector products with $\mtx{X}$ where the number of nonzero entries of $\mtx{X}$ is $\bigO(n)$. This can become prohibitive for large $n$ as computing matrix-vector products is usually the dominant cost.

\subsection{Contributions}
Our first contribution is devising strategies for approximating (or sometimes recovering exactly, e.g. for sparse matrices) functions of sparse matrices and matrices with similar structures, e.g., matrices with only few large entries in each row/column or matrices with rapid off-diagonal decay, which is typically the case when the function $f$ is analytic on the closure of a domain that contains the eigenvalues of $\mtx A$. We provide two strategies, one for functions of sparse matrices or matrices with only few large entries in each row/column and another for functions of banded matrices or matrices with off-diagonal decay. The two strategies are based on approximating the large entries of the matrix in question using matrix-vector products only. This is advantageous if the matrix dimension is too large to store explicitly or the matrix is only available through matrix-vector products. For general sparse matrices and their functions, we use compressed sensing to estimate the large entries of each row/column without needing the sparsity pattern of $\mtx{A}$. For the banded case, our work is analogous to the previous work done in \cite{ColemanMore1983,CPR1974, FrommerSchnimmelSchweitzer2021} but we give a focused treatment of banded matrices.

Our second contribution is designing simple algorithms, Algorithm \ref{alg:gen} (\RS) for sparse matrices with general sparsity pattern and Algorithm \ref{alg:band} (\RB) for banded matrices, with the $2$-norm error bounds based on the decay structure of the matrix. We show that \RS{} converges quickly as long as the matrix has only few large entries in each row/column and that \RB{} converges exponentially with the number of matrix-vector products. Assuming that $f$ is sufficiently smooth so that the Krylov method converges in $\bigO(1)$ iterations when computing $f(\mtx{A})\mtx{b}$, \RS{} runs with complexity $\bigO(n^2\log n)$ for functions of general sparse matrices and \RB{} runs linearly in $n$ with complexity $\bigO(nk^2)$ for functions of $k$-banded matrices. For banded matrices, \RB{} runs with complexity $\approx 4nk^2$ and for functions of banded matrices, \RB{} has the same complexity, $\bigO(nk^2)$, as \cite{CortinovisKressnerMassei2022}. The output of the algorithm is also sparse and has $\bigO(nk)$ nonzero entries where $k$ is the number of nonzero entries recovered per row/column, which can be advantageous in many situations as $f(\mtx{A})$ need not be exactly sparse even if $\mtx{A}$ is sparse.

\subsection{Notation}
Throughout, we write $\flr{x}$ and $\{x\}$ to denote the floor function of $x$ and the fractional part of $x$ respectively.
We use $\norm{\cdot}_p$ where $1\leq p\leq \infty$ to denote the $p$-norm for vectors and matrices and define $\norm{\cdot}_{\max}$ by $\norm{\mtx{A}}_{\max}= \max\limits_{i,j}|A_{ij}|$ where $\mtx{A} \in \R^{m\times n}$ is a matrix. We use $\norm{\vct{v}}_0$ to denote the number of nonzeros in the vector $\vct{v}\in \R^n$. Lastly, we use MATLAB style notation for matrices and vectors. For example, for the $k$th to $(k+j)$th columns of a matrix $\mtx{A}$ we write $A(:,k:k+j)$.

\section{Approximation for functions of general sparse matrices} \label{sec:sparse}
Let $\mtx{A}$ be a diagonalizable sparse matrix with some sparsity pattern. Suppose that $f$ is analytic in a region containing the spectrum of $\mtx{A}$. Then depending on the sparsity pattern of $\mtx{A}$, the entries of $f(\mtx{A})$ decay exponentially away from certain small regions of the matrix. More formally, associate the unweighted directed graph $G(\mtx{A}) = (V,E)$ to $\mtx{A} = [A_{ij}]$ so that $\mtx{A}$ is the adjacency matrix of $G(\mtx{A})$. $V$ is the vertex set consisting of integers from $1$ to $n$ and $E$ is the edge set containing all ordered pairs $(i,j)$ with $A_{ij} \neq 0$. We define the distance $d_G (i,j)$ between vertex $i$ and vertex $j$ to be the shortest directed path from $i$ to $j$.\footnote{Note that if $\mtx{A}$ is not symmetric then $d_G (i,j)\neq d_G (j,i)$ in general.} We then have that there exist constants $C_0>0$ and $0<\lambda_0 <1$ such that 
\begin{equation} \label{eq:fsparsedecay}
        \left|[f(\mtx{A})]_{ij}\right| < C_0 \lambda_0^{d_G (i,j)}
\end{equation} for all $i,j = 1,2,...,n$. For further details see \cite{Benzi2007}. This result tells us that if $d_G (i,j)\gg 1$ then $|[f(\mtx{A})]_{ij}|\ll 1$, which in turn means that depending on the sparsity pattern of $\mtx{A}$, there may only be few large entries in $f(\mtx{A})$ with the rest being negligible. With this observation we can aim to recover the dominating entries of each row/column of $f(\mtx{A})$. We solve this problem using \emph{compressed sensing} by interpreting small entries as noise and the large entries as signals.\footnote{Another possibility would be to identify the nonzero structure using the graph structure $(V,E)$ and $f$. This can quickly become intractable, and in some cases $f(\mtx{A})$ can be sparse despite the graph structure suggesting otherwise, so here we opt for a generic approach.}

Compressed sensing \cite{CandesRombergTao06,Donoho06} is a technique in signal processing for efficiently reconstructing a sparse signal. 
Let $\mtx{Y}\in \R^{n\times s}$ be a so-called \emph{sensing operator} and $\vct{v}\in \R^{n}$ be an unknown sparse vector. Then the problem of recovering $\vct{v}$ using the measurements $\vct{y} = \mtx{Y}^T\vct{v}$ is a combinatorial optimization problem given by
\begin{equation} \label{csprob}
    \min_{\vct{z}\in \R^n} \norm{\vct{z}}_0 \text{ subject to } \vct{y} = \mtx{Y}^T\vct{z}
\end{equation} where $\norm{\vct{z}}_0$ denotes the number of nonzeros in the vector $\vct{z}$. There are two established classes of algorithms, among others, to solve \eqref{csprob} in the compressed sensing literature, namely basis pursuit or $\ell_1$-minimization using convex optimization \cite{CandesTao05}, and greedy methods \cite{BlumensathDaviesRiling12}. For large $n$, convex optimization can become infeasible (especially given that we need to solve~\eqref{csprob} $n$ times to recover the whole matrix), so we use the cheaper option, greedy methods. Greedy methods are iterative and at each iteration attempt to find the correct location and the values of the dominant entries. 

Let $\vct{v}\in \R^n$ be a $k$-sparse vector, that is, a vector with at most $k$ nonzero entries. To recover $\vct{v}$ with theoretical guarantee we typically require at least $s = \bigO(k\log n)$ measurements. In practice, $s = \bigO(k)$ measurements seem to work well \cite{BlanchardTanner15}. Each measurement is taken through vector-vector multiply with $\vct{v}$. The three classes of frequently used sensing operators $\mtx{Y}$ are: a Gaussian matrix $\mtx{\mathcal{N}}$ with entries drawn from i.i.d. $\mathcal{N}(0,s^{-1})$, $s$ columns of the discrete cosine transform matrix $\mtx{\mathcal{C}}$ sampled uniformly at random and a sparse matrix $\mtx{\mathcal{S}}_\xi$ with $\xi$ nonzero entries per row with the support set drawn uniformly at random and the nonzero values drawn uniformly at random from $\pm \xi^{-1/2}$. Here, $\xi$ is the sparsity parameter for the sparse sensing operator.

There are a number of greedy methods available, including Normalized Iterative Hard Thresholding (NIHT) \cite{NIHT}, Hard Thresholding Pursuit (HTP) \cite{HTP} and Compressive Sampling Matching Pursuit (CoSaMP) \cite{CoSaMP}. Among the key differences and the trade-offs between the different methods is the cost per iteration in order to speed up support detection and the number of iterations for better asymptotic convergence rate. Many greedy methods have uniform recovery guarantees if the sensing operator $\mtx{Y}$ has sufficiently small restricted isometry constants \cite{csbook}.\footnote{Under mild assumptions on $s$, the three classes of sensing operators described in the previous paragraph have small restricted isometry constants with high probability.} This means that if we want to recover a matrix then we can use the same sensing operator $\mtx{Y}$ for the recovery of all rows of a matrix, which makes it much more efficient than recovering $n$ rows of a matrix individually using different sensing operators.

In this paper, we use NIHT as it is simple and has a low cost per iteration.\footnote{Other greedy methods such as CoSaMP have similar (or slightly stronger) theoretical guarantees as NIHT, but is generally more expensive; hence making NIHT more preferable in our setting ($k\ll n$ and $n$ problems to solve).} NIHT is an iterative hard thresholding algorithm which performs gradient descent and then hard thresholding at each iteration to update the support and the value of the support set: the $(N+1)$th iterate $\vct{v}^{(N+1)}$ of NIHT is 
  \begin{equation}
    \label{eq:NIHTeqn}
\vct{v}^{(N+1)}  =  H_k\left(\vct{v}^{(N)} +\mu \mtx{Y}\left(\vct{y}-\mtx{Y}^T\vct{v}^{(N)}\right)\right).
  \end{equation}

Here $H_k$ is the hard thresholding operator that sends all but the largest $k$ entries (in absolute value) of the input vector to $0$. The stepsize $\mu$ in NIHT is chosen to be optimal when the support of the current iterate is the same as the sparsest solution.
Assuming that $\mtx{Y}$ satisfies the restricted isometry property (RIP), then at iteration $N$, the output $\vct{v}^{(N)}$ of NIHT satisfies
\begin{equation*}
        \norm{\vct{v}-\vct{v}^{(N)}}_2 \leq 2^{-N}\norm{\vct{v}^{k}}_2 + 8 \epsilon_k,
    \end{equation*} where $\vct{v}^k$ is the best $k$-sparse vector approximation to $\vct{v}$ and
    \begin{equation*}
        \epsilon_k = \norm{\vct{v}-\vct{v}^k}_2 +\frac{1}{\sqrt{k}} \norm{\vct{v}-\vct{v}^k}_1,
    \end{equation*}
    and after at most $N^* = \cei{\log_2 \left(\norm{\vct{v}^k}_2/\epsilon_k\right)}$ iterations, we have
    \begin{equation} \label{eq:NIHTbound}
        \norm{\vct{v}-\vct{v}^{(N^*)}}_2 \leq 9\epsilon_k.
    \end{equation}
Further details about NIHT can be found in \cite{NIHT}. 

The dominating cost for the greedy methods is the matrix-vector product between the sensing operator or the transpose of the sensing operator and the unknown sparse signal, i.e., the cost of $\vct{v}\mapsto \mtx{Y} \vct{v}$ and $\vct{w} \mapsto \mtx{Y}^T\vct{w}$ respectively. If the sensing operator is a Gaussian $\mtx{\mathcal{N}}$ then the cost is $\bigO(ns)$, if $\mtx{Y}$ is a subsampled discrete cosine matrix $\mtx{\mathcal{C}}$ then the cost is $\bigO(n\log n)$ or $\bigO(n\log s)$ if we use the subsampled FFT algorithm \cite{rokhlintygert2008} and lastly, if $\mtx{Y}$ is a sparse matrix $\mtx{\mathcal{S}}_\xi$ then the cost is $\bigO(\xi n)$.

In our context, we are concerned with recovering the entire matrix. We now use compressed sensing row by row or column by column to approximate functions of sparse matrices.
\subsection{Strategy for recovering general sparse matrices}
In this section, we consider matrices $\mtx{B}\in \R^{n\times n}$ that satisfy
\begin{equation} \label{eq:sparsedecay}
    \left|B_{ij}\right| < C_B \lambda_B^{d(i,j)}
\end{equation} where $C_B>0$ and $0<\lambda_B<1$ are constants and $d(i,j) \in \R_{\geq 0} \cup \{+\infty\}$ is some bivariate function that captures the magnitude of the entries of $\mtx{B}$. Functions of sparse matrices $\mtx{B} = f(\mtx{A})$ with a smooth $f$ is a special case of $\mtx{B}$ that satisfies \eqref{eq:fsparsedecay}. Another special case is $\mtx{B} = \mtx{A}$, which recovers the sparse matrix $\mtx{A}$ itself. The zero entries of $\mtx{B}$ would correspond to $d(i,j) = \infty$ in \eqref{eq:sparsedecay}.

We are concerned with recovering the whole matrix. We approximate the dominant entries of each row of $\mtx{B}$ by taking measurements through matrix-vector products $\mtx{B}\vct{y}_i$ where $\vct{y}_i$ is the $i$th column of the sensing operator $\mtx{Y} \in \R^{n\times s}$.\footnote{We can also recover the dominant entries of $\mtx{B}$ column by column using $\mtx{B}^T\mtx{Y}$ if the matrix-vector product $\vct{x}\mapsto \mtx{B}^T\vct{x}$ is available.} Let $\texttt{CS}(\mtx{Y},\mtx{b}_{i}\mtx{Y},k)$ be a function that solves \eqref{csprob} where $\mtx{Y} \in \R^{n\times s}$ is the sensing operator, $\mtx{b}_{i}\mtx{Y} \in \R^{1\times s}$ are the measurements where $\mtx{b}_i$ is the $i$th row of $\mtx{B}$ 
(i.e., $\mtx{b}_{i}\mtx{Y}$ is the $i$th row of $\mtx{BY}$)
and $k$ is the sparsity parameter. For the compressed sensing algorithm $\texttt{CS}$, we can take, for example $\texttt{CS} = $ NIHT. The algorithm to recover the dominant entries of each row of $\mtx{B}$ is given below in Algorithm \ref{alg:gen}, which we call \RS{} 
(Sparse Matrix Recovery Algorithm using Matvecs).

\begin{algorithm}[ht]
  \caption{Sparse Matrix Recovery Algorithm using Matvecs (\RS)}
  \label{alg:gen}
  \begin{algorithmic}[1]
\Require{Matrix-vector product $\texttt{mvp}: \vct{x}\mapsto \mtx{B}\vct{x}$; \newline $\mtx{Y}$ sensing operator; \newline $k$ sparsity parameter; \newline $\texttt{CS}$ function that solves \eqref{csprob} 
from input $(\mtx{Y},\vct{y},k)$
using a compressed sensing algorithm, e.g. NIHT}
\Ensure{$\widehat{\mtx{B}}$ sparse matrix with at most $k$ nonzero entries in each row.}
\vspace{0.5pc}
\State Draw a sensing operator $\mtx{Y}\in \R^{n\times s}$, e.g. Gaussian. \Comment{rec. $s = \lceil 2k\log(n/k) \rceil$}
\State Compute matrix-vector products $\mtx{F}_s = \texttt{mvp}(\mtx{Y})$
\For{$i = 1,2,...,n$}
    \State $\displaystyle \widehat{\mtx{B}}(i,:) = \texttt{CS}(\mtx{Y},\mtx{F}_s(i,:),k)$ \Comment{Using compressed sensing to approximate the rows of $\mtx{B}$}
\EndFor
\end{algorithmic}
\end{algorithm}

The algorithm \RS{} is a straightforward application of compressed sensing techniques used to recover a sparse matrix instead of a vector, particularly when the goal is to recover $\mtx{A}$ itself. We are nonetheless unaware of existing work that spells out explicitly, and \RS{} can be significantly more efficient than alternative approaches \cite{HermanStrohmer2009,PfranderRauhutTanner2008} that vectorize the matrix. 
Since most sensing matrices satisfy the RIP with failure probability exponentially decaying in $s$~\cite{csbook}, by the union bound we can conclude that the whole matrix can be recovered with high probability.

\subsubsection{Complexity of \RS{} (Alg. \ref{alg:gen})}
The complexity of \RS{} is $s$ matrix-vector products with $\mtx{B}$ and $n$ times the complexity of solving the compressed sensing problem. The complexity for solving $n$ compressed sensing problems using NIHT and the subsampled DCT matrix as $\mtx{Y}$ is $\bigO(n^2\log n)$, where we take the number of NIHT iterations to be $\bigO(1)$.

When $\mtx{B} = f(\mtx{A})$ for sparse $\mtx{A}$, we can evaluate $f(\mtx{A})\mtx{Y}$ using Krylov methods. Assuming that $f$ is analytic in a region containing the eigenvalues of $\mtx{A}$,\footnote{It also suffices for $f$ to be approximable to high accuracy with a low-degree polynomial.} the Krylov method converges (to sufficient accuracy) in $\bigO(1)$ iterations, then the complexity of \RS{} is $\bigO(s\cdot nnz(\mtx{A}) + n^2\log n) = \bigO(k\log n \cdot nnz(\mtx{A})+ n^2\log n)$ which consists of $\bigO( s\cdot nnz(\mtx{A})) = \bigO( k\log n \cdot nnz(\mtx{A}))$ flops for computing $s = \bigO(k\log n)$ matrix-vector products with $f(\mtx{A})$ and $\bigO(n^2 \log n)$ flops for solving $n$ compressed sensing problems. Here $nnz(\mtx{A})$ is the number of nonzero entries of $\mtx{A}$.

\subsubsection{Dependency of the compressed sensing solver in \RS{}}

The majority of the computational task in \RS{} (Algorithm \ref{alg:gen}) relies on compressed sensing. There are three main parameter choices in relation to compressed sensing for \RS{}: the sensing operator $\mtx{Y}$, the number of measurements $s$ and the compressed sensing solver \texttt{CS}. 
For the solver, we suggest NIHT as it is the simplest and has a low cost per iteration. For better asymptotic guarantee, but with a higher cost per iteration, we suggest CoSaMP or its variants, e.g. CSMPSP discussed in \cite{BlanchardTanner15}. For the sensing operator $\mtx{Y}$, all three choices mentioned previously are good choices. For robustness, we recommend the Gaussian matrix $\mtx{\mathcal{N}}$ and for speed, we recommend the sparse matrix $\mtx{\mathcal{S}}_{\xi}$ with the sparsity parameter $\xi = \min\{8,s\}$ if a good sparse matrix library is available, otherwise we recommend the subsampled DCT matrix $\mtx{\mathcal{C}}$. Lastly, for the number of measurements $s$ we recommend $s = \lceil 2k\log\left(\frac{n}{k}\right) \rceil$ for theoretical guarantee \cite{Donoho06,CandesRombergTao06} as suggested in \RS{}. However, a constant multiple of the sparsity parameter $k$, say $s = 10k$, is observed to work well in practice \cite{BlanchardTanner15,csbook}.

\subsubsection{A posteriori error estimate for \RS{}} \label{subsec:posterrspamram}

The behaviour of \RS{} can be difficult to grasp as it relies on various parameters that are sometimes implicit, e.g., \eqref{eq:sparsedecay}. Therefore, we can incorporate a posteriori error testing for \RS{} to observe how well it might perform on a particular problem. Let us try to estimate the error using known quantities. We have the sensing operator $\mtx{Y}$, the measurements $\mtx{F}_s\in \R^{n\times s}$ and the output of \RS{}, $\hat{\mtx{B}}$. If the matrix-vector products (line $2$ of \RS{}) have been computed exactly then we have $\mtx{B}\mtx{Y} = \mtx{F}_s$, so we try to estimate the error by determining how close the measurements computed using the output of \RS{}, i.e. $\hat{\mtx{B}}\mtx{Y}$, are to the computed measurements $\mtx{F}_s$. We define
\begin{equation}
    \delta_{RE}^{(S)} := \frac{\norm{\hat{\mtx{B}}\mtx{Y} - \mtx{F}_s}_2}{\norm{\mtx{F}_s}_2}
\end{equation} as the relative error of the computed solution of \RS{} quantified through the measurements. Note that the error estimate $\delta_{RE}^{(S)}$ can be computed a posteriori in $\bigO(nks)$ as $\hat{\mtx{B}}$ has at most $k$ nonzero entries in each row. The error estimate $\delta_{RE}^{(S)}$ can be used as a heuristic for how \RS{} performs. The main motivation for this approach is that a small number of Gaussian measurements estimate the norm well \cite{GrattonTitley-Peloquin2018,mt20}. In our case, we essentially have $\mtx{F}_s = \mtx{B}\mtx{Y}$ and $\hat{\mtx{B}}\mtx{Y}$ giving us $(\mtx{B}-\hat{\mtx{B}})\mtx{Y}$, which is a small number of Gaussian measurements of the absolute error $\mtx{B}-\hat{\mtx{B}}$ if $\mtx{Y}$ is the Gaussian sensing operator.\footnote{For theoretical guarantee, $\mtx{Y}$ needs to be independent from $\mtx{B}$ and $\hat{\mtx{B}}$ \cite{GrattonTitley-Peloquin2018}, however the output of \RS{}, $\hat{\mtx{B}}$ is dependent on $\mtx{Y}$. Nonetheless, the dependency is rather weak so $\delta_{RE}^{(S)}$ should be a good error estimate in practice. For theoretical guarantee, one can take, say $5$ extra measurements for error estimation.} Therefore $\delta_{RE}$ provides a good guidance for how small the error is for \RS{}. For the other sensing operators, the behaviour is similar \cite{mt20}. Later, we show in a numerical experiment (Figure \ref{fig:syneg}) in Section \ref{sec:numill} that $\delta_{RE}^{(S)}$ can estimate the performance of \RS{} well.

\subsubsection{Adaptive methods} There are adaptive methods that reduce the number of matrix-vector products required for sparse recovery, for example \cite{adaptive1,adaptive3,KrahmerWard2014,adaptive2} by making adaptive measurements based on previous measurements. However, in our context we are concerned with recovering the whole matrix and it is difficult to make adaptive measurements that account for every row/column of $\mtx{B}$ when we only have access to $\mtx{B}$ using matrix-vector products. A naive strategy is to make adaptive measurements for each row/column, but this requires at least $\bigO(n)$ matrix-vector products, which is computationally infeasible and futile as we can recover $\mtx{B}$ using $n$ measurements using the $n\times n$ identity matrix.

\subsubsection{Shortcomings}
There are some limitations with \RS. First, the entries of $f(\mtx{A})$ are only negligible when $d(i,j)\gg 1$. Indeed, there are matrices for which the sparsity pattern implies $d(i,j) = \bigO(1)$ for all or most $i,j$. An example is the arrow matrix which has nonzero entries along the first row, first column and the main diagonal. The sparsity pattern of the arrow matrix makes $d(i,j)\leq 2$, which can make \RS{} behave poorly as the a priori estimate from \eqref{eq:fsparsedecay} tells us that $f(\mtx{A})$ may have no entries that are significantly larger than others. A similar issue has been discussed in \cite{Dasarathyetal2015}, where the authors assume that the matrix we want to approximate is $d$-distributed, that is, the matrix has at most $d$ nonzero entries along any row or column. We will use a similar notion, which will be described in Subsection \ref{subsec:analgen} to analyze \RS.

\RS{} requires $k$ as an input. This may be straightforward from the sparsity pattern of $\mtx{A}$ and the smoothness of $f$, but in some cases a good choice of $k$ can be challenging. Moreover, when most rows of $\mtx{B}$ are very sparse but a few are much denser, $k$ would need to be large enough for the densest row to recover the whole matrix $\mtx{B}$. However, even with a small $k$, the sparse rows will be recovered.

Another downside is in the complexity of \RS{} for functions of sparse matrices. The cost of \RS{} is $\bigO(k\log n \cdot nnz(\mtx{A})+n^2\log n)$, which is dominated by the cost of solving $n$ compressed sensing problems, $\bigO(n^2\log n)$ but computing $f(\mtx{A})\mtx{I}_n$ where $\mtx{I}_n$ is the $n\times n$ identity matrix costs $\bigO(n\cdot nnz(\mtx{A}))$,
which is cheaper if $nnz(\mtx{A})=\bigO(n)$
and approximates $f(\mtx{A})$ with higher accuracy. However, if matrix-vector products with $\mtx{B} = f(\mtx{A})$ start dominating the cost of \RS, or when $\mtx{A}$ is not so sparse with $nnz(\mtx{A})> n\log n$, but $f(\mtx{A})$ is approximately sparse, then there are merits in using \RS. It is also worth noting that storing the measurements $f(\mtx{A})\mtx{Y}$ and only solving the compressed sensing problem to approximate specific row(s) when necessary would save computational time. This may be useful when we only need to approximate some subset of rows of $f(\mtx{A})$, but we do not know the indices in advance.

\subsection{Analysis of \RS{} (Alg. \ref{alg:gen})} \label{subsec:analgen}
The class of matrices $\mtx{B}$ that satisfy the decay bound \eqref{eq:sparsedecay} can be difficult to analyze as the behaviour of $d(i,j)$ can give non-informative bounds in \eqref{eq:sparsedecay}, i.e., there may be no entries that are small. For functions of matrices, the arrow matrix is an example, which has the sparsity pattern that makes $d_G(i,j)\leq 2$. When $\mtx{B}$ has entries that all have similar magnitudes, \RS{} fails to compute a good approximation to $\mtx{B}$. 

We now define a notion similar to $d$-distributed sparse matrices in \cite{Dasarathyetal2015} to analyze \RS. Let $d \in \Z^+$ be a distance parameter and define
\begin{equation*}
    k_d = \max_{1\leq i\leq n} \left\lvert \left\{ j: d(i,j) \leq d\right\} \right\rvert.
\end{equation*} When $\mtx{B} = f(\mtx{A})$, $k_d$ measures the maximum number of vertices that are at most distance $d$ away from any given vertex. $k_d$ also measures the combined sparsity of $\mtx{I}_n,\mtx{A},...,\mtx{A}^d$ because the cardinality of the union of the support locations in a fixed row in each of $\mtx{I}_n,\mtx{A},...,\mtx{A}^d$ will be at most $k_d$. If $k_d = \bigO(1)$ for $d$ large enough such that $f$ can be approximated by a polynomial of degree $d-1$ on $\mtx{A}$'s spectrum, then each row of $f(\mtx{A})$ would only have $\bigO(1)$ dominating entries. Using the definition of $k_d$, we now analyze \RS{} in Theorem \ref{thm:gen}.

\begin{theorem}[Algorithm \ref{alg:gen}] \label{thm:gen}
    Let $\mtx{B}\in\mathbb{R}^{n\times n}$ be a matrix satisfying
    \begin{equation} 
    |B_{ij}| < C_B \lambda_B^{d(i,j)} \tag{\ref{eq:sparsedecay}}
\end{equation} for constants $C_B>0$, $1>\lambda_B>0$ and $d(i,j) \in \R_{\geq 0} \cup \{+\infty\}$ is some bivariate function that captures the magnitude of the entries of $\mtx{B}$.
    Then the output of \RS, $\widehat{\mtx{B}}$, satisfies
    \begin{equation} \label{sparsebound2}
        \norm{\widehat{\mtx{B}}-\mtx{B}}_2 \leq 9C_B \left( n + \frac{n^{3/2}}{\sqrt{k_d}} \right) \lambda_B^{d+1}
    \end{equation}
    where 
    \begin{equation*}
        k_d = \max_i \left\lvert \left\{ j: d(i,j) \leq d\right\} \right\rvert
    \end{equation*} is the sparsity parameter in Algorithm \ref{alg:gen} with the distance parameter $d$ and the compressed sensing problem was solved row by row using NIHT  until \eqref{eq:NIHTbound} was satisfied.
\end{theorem}
\begin{proof}
    Define $\mtx{E} = \widehat{\mtx{B}}-\mtx{B}$. 
    Let $E(i,:)$ be the $i$th row of $\mtx{E}$. Then from \eqref{eq:NIHTbound} we get
    \begin{equation*}
        \norm{E(i,:)}_2 \leq 9\left(\norm{\vct{b}_i - \vct{b}_{i}^{k_d}}_2 +\frac{1}{\sqrt{k_d}} \norm{\vct{b}_i - \vct{b}_i^{k_d}}_1\right)
    \end{equation*} where $\vct{b}_{i}$ is the $i$th row of $\mtx{B}$ and $\vct{b}_{i}^{k_d}$ is the best $k_d$-term approximation to $\vct{b}_{i}$. Since $\vct{b}_{i}^{k_d}$ is the best $k_d$-term approximation to $\vct{b}_{i}$, we have 
    \begin{equation*}
        \norm{\vct{b}_{i} - \vct{b}_{i}^{k_d}}_{\max} \leq C_B \lambda_B^{d+1}
    \end{equation*} using \eqref{eq:sparsedecay}.
    Therefore
    \begin{equation*}
        \norm{\vct{b}_{i} - \vct{b}_{i}^{k_d}}_2 \leq \sqrt{n-k_d} C_B  \lambda_B^{d+1}, \qquad 
        \norm{\vct{b}_{i} - \vct{b}_{i}^{k_d}}_1 \leq (n-k_d)C_B\lambda_B^{d+1}.
    \end{equation*}
    Therefore
    \begin{align*}
        \norm{\mtx{E}}_2  &\leq \sqrt{n} \max_{1\leq i \leq n}\norm{E(i,:)}_2 \\
        &\leq 9\sqrt{n} \max_{1\leq i \leq n} \norm{\vct{b}_{i} - \vct{b}_{i}^{k_d}}_2 +\frac{1}{\sqrt{k_d}} \norm{\vct{b}_{i} - \vct{b}_{i}^{k_d}}_1 \\
        &\leq 9\sqrt{n} \left( C_B\sqrt{n-k_d} \lambda_B^{d+1} + C_B\frac{n-k_d}{\sqrt{k_d}} \lambda_B^{d+1} \right) \\
        &< 9C_B \left( n + \frac{n^{3/2}}{\sqrt{k_d}} \right) \lambda_B^{d+1}.
    \end{align*}
\end{proof}
\begin{remark} \*
\begin{enumerate}
    \item The bound in Theorem \ref{thm:gen} can be pessimistic as we can replace the $2$-norm bound by the Frobenius norm bound in \eqref{sparsebound2}. This can be shown by noting
    \begin{align*}
        \norm{\mtx{E}}_F  &\leq \sqrt{\sum_{1\leq i \leq n}\norm{E(i,:)}_2^2} \leq \sqrt{n} \max_{1\leq i \leq n}\norm{E(i,:)}_2 < 9C_B \left( n + \frac{n^{3/2}}{\sqrt{k_d}} \right) \lambda_B^{d+1}.
    \end{align*}
    \item The bound in \eqref{sparsebound2} can be improved if for each $i$, the nonzero entries of $\vct{b}_i - \vct{b}_i^{k_d}$ decay exponentially. This implies
    \begin{equation*}
        \norm{\vct{b}_i - \vct{b}_i^{k_d}}_{1} \lesssim \lambda_B^{d+1},\qquad 
        \norm{\vct{b}_i - \vct{b}_i^{k_d}}_{2} \lesssim \lambda_B^{d+1}
    \end{equation*}and the bound improves to
    \begin{equation*}
        \norm{\mtx{E}}_2 \lesssim \sqrt{n}\left(1+ \frac{1}{\sqrt{k_d}} \right) \lambda_B^{d+1}.
    \end{equation*}
    \item In general, the constants $C_B,\lambda_B$ and the bivariate function $d(i,j)$ are either unknown, implicit or too expensive to compute in applications. For matrix functions, Theorem \ref{thm:gen} relies on the decay bound by Benzi and Razouk \cite{Benzi2007}, which contains constants that are implicit as they depend on the function $f$ and the spectral property of the matrix $\mtx{A}$. Therefore, without a result that includes explicit constants, Theorem \ref{thm:gen} cannot serve as a priori estimate for Algorithm \ref{alg:gen}. However, the analysis in Theorem \ref{thm:gen} can still serve as a heuristic for when Algorithm \ref{alg:gen} will be useful.
\end{enumerate}
\end{remark}


\section{Approximation for functions of banded matrices} \label{sec:band}
Let $\mtx{A}$ be a diagonalizable banded matrix with upper bandwidth $k_1$ ($A_{ij} = 0$ if $j-i>k_1$) and lower bandwidth $k_2$ ($A_{ij} = 0$ if $i-j>k_2$).
This is a more restrictive case than Section~\ref{sec:sparse}, and so \RS{} can be used; however, in this case the sparsity pattern is predictable, and hence more efficient algorithms are possible.

 Suppose that $f$ is a function that is analytic in the region containing the spectrum of $\mtx{A}$. We then have that the entries of $f(\mtx{A})$ decay exponentially away from the main diagonal. More specifically, we have that there exists constants $C_1,C_2>0$ and $0<\lambda_1,\lambda_2<1$ dependent on $\mtx{A}$ and $f$ such that for $i< j$,
\begin{equation} \label{eq:ubandbound}
        \left|[f(\mtx{A})]_{ij}\right| < C_1 \lambda_1^{j-i}
    \end{equation} and for $i\geq j$,
    \begin{equation} \label{eq:lbandbound}
        \left|[f(\mtx{A})]_{ij}\right| < C_2 \lambda_2^{i-j}.
    \end{equation}
The proof of this result is based on polynomial approximation of $f$, together with the simple fact that powers of $\mtx{A}$ are banded if $\mtx{A}$ is. The exponents satisfy $\log \lambda_1^{-1} \propto k_1^{-1}$ and $\log \lambda_2^{-1} \propto k_2^{-1}$. For more details see \cite{Benzi2007,BenziSimoncini2015, PozzaSimoncini2019}. Unlike the general sparse case, here the location of the large entries is known, i.e., near the main diagonal.

We will exploit the exponential off-diagonal decay given above in \eqref{eq:ubandbound} and \eqref{eq:lbandbound} to efficiently recover functions of banded matrices using matrix-vector products. We first design a way to recover banded matrices using matrix-vector products similarly to how banded Jacobian matrices were recovered in \cite{ColemanMore1983,CPR1974} and then use the strategy used for the banded matrices along with the exponential off-diagonal decay to approximate functions of banded matrices and matrices with rapid off-diagonal decay using matrix-vector products. The idea of using the same strategy for matrices with rapid off-diagonal decay was suggested in \cite{BekasKokiopoulouSaad2007} and was later applied to functions of banded matrices in \cite{FrommerSchnimmelSchweitzer2021}.

\subsection{Strategy for recovering banded matrices} \label{sec:bandrecovery}
Let $\mtx{A}\in \R^{n\times n}$ be a banded matrix with upper bandwidth $k_1$ and lower bandwidth $k_2$. $\mtx{A}$ has at most $1+k_1+k_2$ nonzero entries in each row and column, which lie consecutively near the main diagonal. Therefore if we recover the $1+k_1+k_2$ consecutive nonzero entries near the main diagonal in each row or column then we can recover $\mtx{A}$. This motivates us to use a simple set of $s$ vectors, which to the authors' knowledge, originates from \cite{CPR1974} in the context of recovering matrices using matrix-vector products. The $s$ vectors are 
\begin{equation} \label{eq:repeatedid}
    \mtx{I}_{n}^{(s)} = [\mtx{I}_s,\mtx{I}_s,...,\mtx{I}_s,\mtx{I}_s(:,1:\{n/s\}\cdot s)]^T \in \R^{n\times s}
\end{equation} where $s<n$ is a positive integer and $\mtx{I}_s(:,1:\{n/s\}\cdot s)$ is the first few columns of $\mtx{I}_s$ to ensure $\mtx{I}_{n}^{(s)}$ has exactly $n$ rows. This set of $s$ vectors capture $s$ consecutive nonzero entries of $\mtx{A}$ in each row when right-multiplied to $\mtx{A}$ if the other $(n-s)$ entries are all zeros. By choosing $s = 1+k_1+k_2$ we are able to recover every nonzero entry of $\mtx{A}$. We illustrate this idea with an example.

\paragraph{Example} Let $\mtx{A}_6 \in \R^{6\times 6}$ be a banded matrix with upper bandwidth $2$ and lower bandwidth $1$. Then the matrix-vector products of $\mtx{A}_6$ with $\mtx{I}_6^{(1+2+1)}$ give
\begin{equation*}
    \mtx{A}_6 \mtx{I}_6^{(4)} = \begin{bmatrix}
        b_1 & c_1 & d_1 & 0 & 0 & 0 \\
        a_1 & b_2 & c_2 & d_2 & 0 & 0 \\
        0 & a_2 & b_3 & c_3 & d_3 & 0 \\
        0 & 0 & a_3 & b_4 & c_4 & d_4 \\
        0 & 0 & 0 & a_4 & b_5 & c_5 \\
        0 & 0 & 0 & 0 & a_5 & b_6 \\
    \end{bmatrix} \begin{bmatrix}
        1 & 0 & 0 & 0  \\
        0 & 1 & 0 & 0  \\
        0 & 0 & 1 & 0  \\
        0 & 0 & 0 & 1  \\
        1 & 0 & 0 & 0  \\
        0 & 1 & 0 & 0  \\
    \end{bmatrix} = \begin{bmatrix}
        b_1 & c_1 & d_1 & 0  \\
        a_1 & b_2 & c_2 & d_2  \\
        d_3 & a_2 & b_3 & c_3  \\
        c_4 & d_4 & a_3 & b_4  \\
        b_5 & c_5 & 0 & a_4  \\
        a_5 & b_6 & 0 & 0  \\
    \end{bmatrix}.
\end{equation*}

We observe in the example that every nonzero entry of $\mtx{A}_6$ in each row has been copied exactly once up to reordering onto the same row of $\mtx{A}_6\mtx{I}_6^{(4)}$. Since we know $\mtx{A}_6$ has upper bandwidth $2$ and lower bandwidth $1$, we can reorganize the entries of $\mtx{A}_6\mtx{I}_6^{(4)}$ to recover $\mtx{A}_6$ exactly. This idea can be generalized as follows \cite{ColemanMore1983,CPR1974}.\footnote{The work in \cite{ColemanMore1983,CPR1974} is based on partitioning/coloring of the adjacency graph of a banded matrix. In this work, we show the vectors more explicitly.}

\begin{proposition} \label{thm:bband}
    Let $\mtx{A} \in \R^{n\times n}$ be a banded matrix with upper bandwidth $k_1$ and lower bandwidth $k_2$. Then $\mtx{A}$ can be exactly recovered using $1+k_1+k_2$ matrix-vector products using $\mtx{I}_n^{(1+k_1+k_2)}$ as in \eqref{eq:repeatedid}.
\end{proposition}
\begin{proof}\footnote{The proof of Proposition \ref{thm:bband} can be simplified by arguing that we need at most $1+k_1+k_2$ colors to color the columns of $\mtx{A}$ such that all the columns of the same color have disjoint support. Now, probing all the columns of the same color simultaneously yields the result. In this work, we provide a lengthier proof of Proposition \ref{thm:bband} to make certain observations that lay the groundwork for our strategy to recover functions of banded matrices in Section \ref{subsec:stratband}.}
    Let $s = 1+k_1+k_2$ and for $1\leq i \leq n$ and $1\leq j \leq s$, define
    \begin{align*}
        t_{ij} &= \argmin_{\substack{1\leq j+st \leq n \\ t\in \Z}}\left|j+st-\left(i+(s-1)\left(\frac{1}{2}-\frac{k_1}{k_1+k_2}\right)\right)\right| \\
        &= \argmin_{\substack{1\leq j+st \leq n \\ t\in \Z}}\left|j+st-i-\frac{k_2-k_1}{2}\right|.
    \end{align*} Equivalently, $t_{ij}$ is the unique integer satisfying
    \begin{equation*}
        k_1 \geq j+st_{ij}-i \geq -k_2, \text{ and } 1\leq j+st_{ij} \leq n.
    \end{equation*}
   Then the $(i,j)$-entry of $\mtx{A}\mtx{I}_n^{(s)}$ is given by
    \begin{equation} \label{eq:keyterm}
        \left[\mtx{A}\mtx{I}_n^{(s)}\right]_{ij} = \sum_{\substack{1\leq j+st \leq n \\ t\in \Z}} A_{i,j+st} = A_{i,j+st_{ij}}
    \end{equation} since $A_{k\ell} = 0$ for all $k_2 < k-\ell$ and $k_1 < \ell-k$.
    
    Now we show that for all $1\leq i,r \leq n$ with $k_1\geq r-i \geq -k_2$, there is a unique entry of $\mtx{A}\mtx{I}_{n}^{(s)}$ equal to $A_{ir}$, i.e., we can use the entries of $\mtx{A}\mtx{I}_{n}^{(s)}$ to recover every nonzero entry of $\mtx{A}$. Let $1\leq i,r \leq n$ be integers such that $k_1 \geq r-i \geq -k_2$ and $1\leq j \leq s$ be a unique integer such that $r = st^*+j$ for some integer $t^*$. Then 
    \begin{equation*}
        \left[\mtx{A}\mtx{I}_n^{(s)}\right]_{ij} = \sum_{\substack{1\leq j+st \leq n\\ t\in \Z}} A_{i,j+st} = A_{i,j+st^{*}} = A_{ir}.
    \end{equation*} Therefore $t^* = t_{ij}$ and for each $k_1 \geq r-i \geq -k_2$, there is a unique entry $\left[\mtx{A}\mtx{I}_{n}^{(s)}\right]_{ij}$ equal to $A_{ir}$ with $j = r-st_{ij}$. Therefore $\mtx{A}$ can be exactly recovered using $1+k_1+k_2$ matrix-vector products using $\mtx{I}_n^{(1+k_1+k_2)}$.
\end{proof}

We now extend this idea to approximate functions of banded matrices and matrices with off-diagonal decay.

\subsection{Strategy for recovering functions of banded matrices} \label{subsec:stratband}
To accurately approximate $f(\mtx{A})$ using matrix-vector products, we want to find a set of vectors when multiplied with $f(\mtx{A})$, approximate the large entries of $f(\mtx{A})$. We use the same set of vectors as for the banded case, namely $\mtx{I}_n^{(s)}$. We assume that $s$ is an odd integer $s = 2s_0+1$ throughout this section. We will see that this set of simple vectors can approximate functions of banded matrices. This is the same as the observation made in \cite{BekasKokiopoulouSaad2007} and the work done later in \cite{FrommerSchnimmelSchweitzer2021}, which are based on graph coloring of the adjacency graph of a matrix with off-diagonal decay.

In this section, we consider matrices $\mtx{B} \in \R^{n\times n}$ with exponential\footnote{The analysis can be carried out similarly for any other type of decay, for example, algebraic.} off-diagonal decay with $\mtx{B} = f(\mtx{A})$ for banded $\mtx{A}$ being the special case. For simplicity we consider matrices with the same decay rate on both sides of the diagonal, that is, $\mtx{B}$ satisfies
\begin{equation} \label{eq:symdecayrate}
    \left|\mtx{B}_{ij}\right| < C \lambda^{|i-j|}
\end{equation} for some constant $C>0$ and $0<\lambda<1$. The results in this section can easily be extended to the more general case, that is, the class of matrices with entries satisfying \eqref{eq:ubandbound} and \eqref{eq:lbandbound}, i.e., the decay rate along the two sides of the diagonal are different. See Subsection \ref{sec:difdecayrate} for a further discussion.

The key idea is to revisit \eqref{eq:keyterm} in the proof of Proposition \ref{thm:bband}. For $1\leq i\leq n$ and $1\leq j \leq s$, recall that $t_{ij}$ is defined by
\begin{equation*}
    t_{ij} = \argmin_{\substack{1\leq j+st \leq n \\ t\in \Z}}\left|j+st-i\right|
\end{equation*}
when the decay rate is the same on both sides of the diagonal.\footnote{When the decay rate is different on each side of the diagonal we can carry out the analysis in a similar way by defining $t_{ij}$ differently. See Section \ref{sec:difdecayrate}. } $t_{ij}$ is an integer such that $i$ and $j+st_{ij}$ are as close as possible, i.e., the index closest to the main diagonal. \eqref{eq:keyterm} gives us
\begin{equation*}
    \left[\mtx{B}\mtx{I}_n^{(s)}\right]_{ij} = \sum_{\substack{1\leq j+st \leq n \\ t\in \Z}} B_{i,j+st}.
\end{equation*}
By \eqref{eq:symdecayrate}, the sum is dominated by $B_{i,j+st_{ij}}$ and the remaining terms decay exponentially. Therefore we can view this as 
\begin{equation*}
    \left[\mtx{B}\mtx{I}_n^{(s)}\right]_{ij} = \underbrace{B_{i,j+st_{ij}}}_{\text{Dominant term}}+\underbrace{\sum_{\substack{1\leq j+st \leq n \\ t\neq t_{ij}\in \Z}} B_{i,j+st},}_{\text{Sum of exponentially decaying terms}}
\end{equation*} which motivates us to use $\left[\mtx{B}\mtx{I}_n^{(s)}\right]_{ij}$ as an approximation for $B_{i,j+st_{ij}}$.

Now create a matrix $\widehat{\mtx{B}}$ by 
\begin{equation} \label{reconeq}
    \widehat{B}_{i,j+st_{ij}} = \left[\mtx{B}\mtx{I}_n^{(s)}\right]_{ij} = \sum_{\substack{1\leq j+st \leq n \\ t\in \Z}} B_{i,j+st}
\end{equation} for $|j+st_{ij}-i| \leq s_0$ where $1\leq i\leq n$ and $1\leq j\leq s$ and zero everywhere else. Notice that by construction, $\widehat{\mtx{B}}$ is an $s_0$-banded matrix. Roughly, the nonzero entries of the $i$th row of $\widehat{\mtx{B}}$ will be, with some small error, the largest entries of the $i$th row of $\mtx{B}$ up to reordering. The entries of $\widehat{\mtx{B}}$ approximate those of $\mtx{B}$ well because 
\begin{equation} \label{eq:enterr}
   \left\lvert \widehat{B}_{i,j} - B_{i,j}\right\rvert = \left\lvert\sum_{\substack{1\leq j+st \leq n \\  t \neq 0 \in \Z}} B_{i,j+st}\right\rvert \approx C\lambda^{s_0}
\end{equation} for $|j-i| \leq s_0$ and $1\leq i,j \leq n$. The largest error incurred from \eqref{eq:enterr} is at most $\bigO\left(\left|B_{i,i\pm s_0}\right|\right)$ for the $i$th row and the zero entries of $\widehat{\mtx{B}}$, i.e. the entries not covered by \eqref{eq:enterr}, incur error of at most $\bigO\left(\left|B_{i,i\pm s_0}\right|\right)$ also. Therefore,
\begin{equation} \label{erranal}
    \norm{\widehat{\mtx{B}}-\mtx{B}}_{\max} = \bigO\left(\max_{1\leq i\leq n}\left|B_{i,i\pm s_0}\right|\right).
\end{equation} This qualitative max-norm error bound of the approximant $\widehat{\mtx{B}}$ will be used to derive the $2$-norm error bound in Section \ref{subsec:bandanal}. The precise details will be laid out in Theorem \ref{thm:band}.

Figure \ref{fig:decay} illustrates the error analysis given above for $\mtx{B} = f(\mtx{A})$ where $f$ is the exponential function $e^x$ and $\mtx{A}\in \R^{1000\times 1000}$ is a symmetric $2$-banded matrix with entries drawn i.i.d. from $\mathcal{N}(0,n^{-1})$. This matrix has real eigenvalues with $\lambda(A)\in [-0.1919,0.2011]$.
The matrix-vector products $f(\mtx{A})\mtx{I}_n^{(s)}$ were evaluated using the polynomial Krylov method with $4,6$ and $8$ iterations (corresponding to polynomial approximation of degrees $3,5,7$ respectively), which give approximations to $f(\mtx{A})\mtx{I}_n^{(s)}$ of varying accuracy. In Figure \ref{fig1a}, we see the exponential decay in the entries of $f(\mtx{A})$ and in Figure \ref{fig1b}, we see that \eqref{erranal} captures the decay rate until the error is dominated by the error from computing matrix-vector products using the Krylov method. The stagnation of the $\max\limits_{i = 1,2,...,n} \left|[f(\mtx{A})]_{i,i\pm (s-1)/2}\right|$ curve is caused by machine precision.
\begin{figure}[!ht]
\subfloat{\label{fig1a}\includegraphics[scale = 0.45]{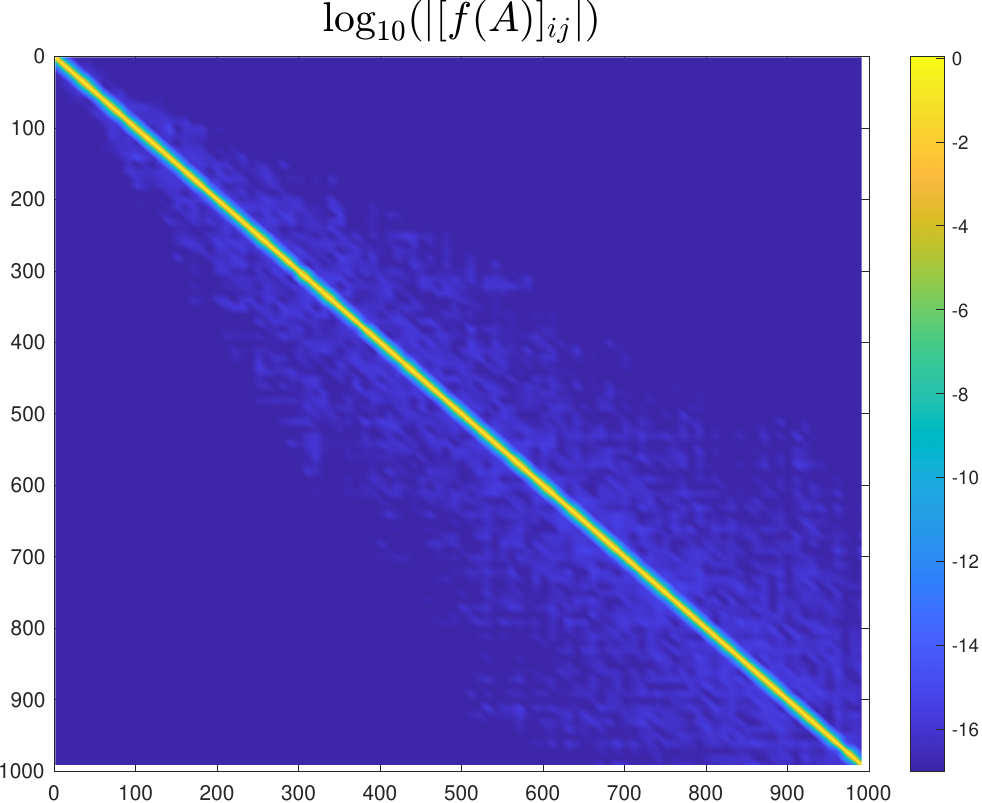}} \hfill
\subfloat{\label{fig1b}\includegraphics[scale = 0.45]{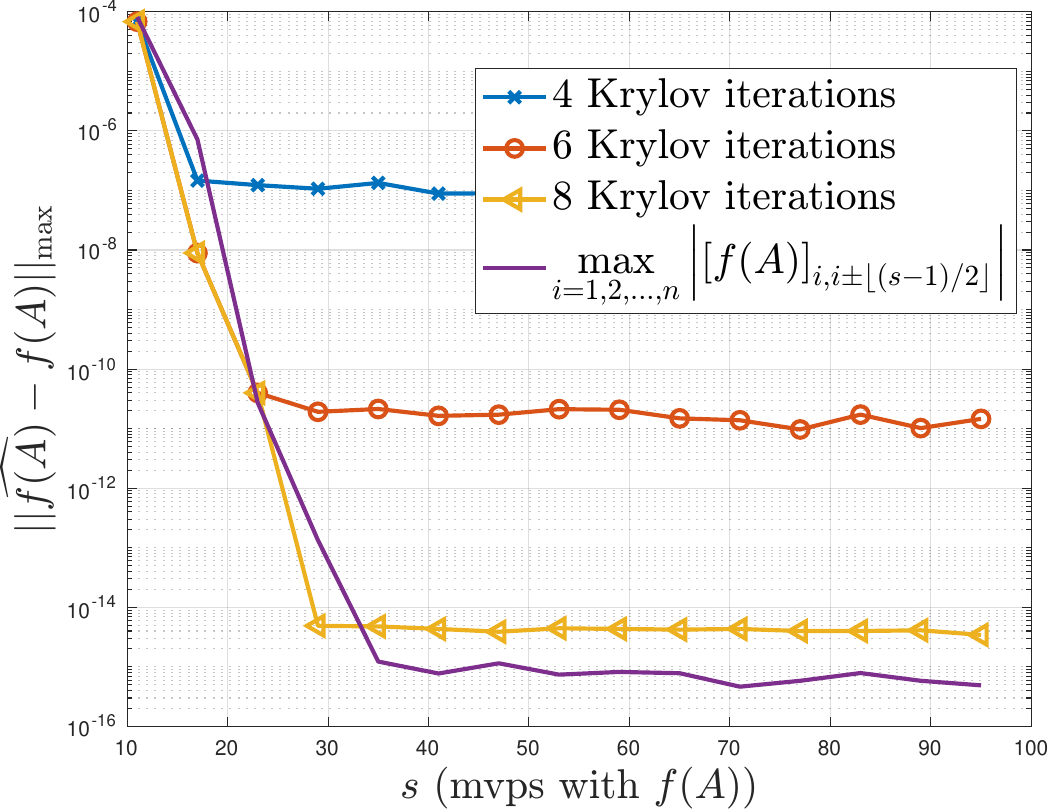}}
\centering 
\caption{$f$ is the exponential function $e^x$ and $\mtx{A}\in \R^{1000\times 1000}$ is a symmetric $2$-banded matrix. The left figure shows the exponential decay in the entries of $f(\mtx{A})$ and the right figure shows that the error analysis in \eqref{erranal} captures the decay rate until the error is dominated by the error from evaluating matrix-vector products using the Krylov method. The stagnation of the $\max\limits_{i = 1,2,...,n} \left|[f(\mtx{A})]_{i,i\pm (s-1)/2}\right|$ curve is caused by machine precision. The exponential decay in the max-norm error will be used to derive the $2$-norm error bound in Section \ref{subsec:bandanal}.}
\label{fig:decay}
\end{figure}

We now devise an algorithm based on the above analysis to approximate matrices that satisfy the exponential off-diagonal decay. The algorithm is given in Algorithm \ref{alg:band}, which we call \RB{}  (Banded Matrix Recovery Algorithm using Matvecs).

\begin{algorithm}[ht]
  \caption{Banded Matrix Recovery Algorithm using Matvecs (\RB)}
  \label{alg:band}
  \begin{algorithmic}[1]
\Require{Matrix-vector product function handle $\texttt{mvp}: \vct{x}\mapsto \mtx{B}\vct{x}$; \newline $s = 2s_0+1$ the number of matrix-vector products; (rec. $s_0 = 10k$ for functions of $k$-banded matrices)}
\Ensure{$\widehat{\mtx{B}}$, $s_0$-banded approximation to $\mtx{B}$;}
\vspace{0.5pc}
\State Compute matrix-vector products $\mtx{B}_s = \texttt{mvp}\left(\mtx{I}_{n}^{(s)}\right) \in \R^{n\times s}$, where $\mtx{I}_n^{(s)}$ is as in \eqref{eq:repeatedid}
\State Set the entries of $\widehat{\mtx{B}}$ using $\mtx{B}_s$ and \eqref{reconeq}
\end{algorithmic}
\end{algorithm}

\subsubsection{Complexity of \RB{} (Alg. \ref{alg:band})}
The complexity of \RB{} is $s$ matrix-vector products with $\mtx{B}$. In the case $\mtx{B} = f(\mtx{A})$ for a $k$-banded matrix $\mtx{A}$, we can evaluate $f(\mtx{A})\mtx{I}_n^{(s)}$ using Krylov method. Assuming that $f$ is analytic in a region containing the eigenvalues of $\mtx{A}$, the polynomial Krylov method converges (to sufficient accuracy) in $\bigO(1)$ iterations and $s = 2ck+1$ for some constant $c$, then the complexity of \RB{} is $\bigO(snk) = \bigO(nk^2)$.

If we only have access to the matrix using its transpose, i.e. $\vct{x} \mapsto \mtx{A}^T\vct{x}$ (or $\vct{x} \mapsto \mtx{B}^T\vct{x}$), then we can take $f(\mtx{A}^T)\mtx{I}_n^{(s)}$ (or $\mtx{B}^T\mtx{I}_n^{(s)}$) and use the rows of $f(\mtx{A}^T)\mtx{I}_n^{(s)}$ (or $\mtx{B}^T\mtx{I}_n^{(s)}$) to recover the columns of $f(\mtx{A})$ (or $\mtx{B}$).

\subsubsection{Error estimate for \RB{}} \label{subsec:errestband}
The qualitative behaviour of \RB{} is better understood than \RS{}. Since $\mtx{B}$ is expected to have an off-diagonal decay, an increase in $s$ would improve the quality of the estimate $\hat{\mtx{B}}$ in \RB{}. In the context of matrix functions, the speed of the off-diagonal decay of $f(\mtx{A})$ is positively correlated to the smoothness of $f$ and negatively correlated to the size of the bandwidth of $\mtx{A}$.
However, unlike \RS{}, \RB{} estimates the entries of $\mtx{B}$ using the entries of the measurements $\mtx{B}_s$. Therefore, by construction, $\norm{\hat{\mtx{B}}\mtx{I}_n^{(s)}-\mtx{B}_s}_2 = 0$ and we cannot use the same strategy as \RS{} in Section \ref{subsec:posterrspamram}. However, we can estimate the error by making a small number of extra Gaussian measurements of $\mtx{B}$. We can estimate the error using
\begin{equation*}
    \delta_{RE}^{(B)}:=\frac{\norm{\hat{\mtx{B}}\mtx{X} - \mathtt{mvp}(\mtx{X})}_2}{\norm{\mathtt{mvp}(\mtx{X})}_2}
\end{equation*} where $\mtx{X}\in \R^{n\times 5}$ is a Gaussian sensing operator. As previously mentioned, making a small constant number of Gaussian measurements, here we take $5$, usually suffices in practice; see \cite{GrattonTitley-Peloquin2018,mt20}. We recommend computing $ \delta_{RE}^{(B)}$ when the decay behaviour of $\mtx{B}$ is unknown or in the context of matrix functions, when the smoothness of $f$ or the bandwidth of the underlying matrix $\mtx{A}$ is unknown. Later in Section \ref{sec:numill}, we illustrate the error estimate $\delta_{RE}^{(B)}$ in a numerical experiment; see Figure \ref{fig:syneg}.

\subsubsection{Analysis of \RB (Alg. \ref{alg:band})} \label{subsec:bandanal}
Here we give theoretical guarantees for \RB{}. The analysis for \RB{} is easier and it gives a stronger bound than \RS{} as the precise location of the dominant entries and the qualitative behaviour of the decay is known. The proof is motivated by Proposition 3.2 in \cite{Benzi2007}.
\begin{theorem}[Algorithm \ref{alg:band}] \label{thm:band}
    Let $\mtx{B}$ be a matrix satisfying
    \begin{equation}
    \left|\mtx{B}_{ij}\right| < C \lambda^{|i-j|} \tag{\ref{eq:symdecayrate}}
    \end{equation} where $C>0$ and $0<\lambda<1$ are constants.
    Then the output of \RB, $\widehat{\mtx{B}}$, satisfies
    \begin{equation*}
        \norm{\widehat{\mtx{B}} - \mtx{B}}_2 \leq \frac{4C\lambda}{1-\lambda} \lambda^{s_0}
    \end{equation*} where $s = 2s_0+1$.
\end{theorem}
\begin{proof}
    Define $\mtx{E} = \widehat{\mtx{B}} - \mtx{B}$. We get
    \begin{equation*}
        \norm{\widehat{\mtx{B}} - \mtx{B}}_2 = \norm{\mtx{E}}_2 \leq \sqrt{\norm{\mtx{E}}_1 \norm{\mtx{E}}_\infty}
    \end{equation*} using a special case of H\"older's inequality {\cite[\S 2.3.3]{matcomp}}. We now bound $\norm{\mtx{E}}_1$ and $\norm{\mtx{E}}_\infty$ using \eqref{eq:symdecayrate}. By the construction of $\widehat{\mtx{B}}$, the $i$th row of $\mtx{E}$ satisfies
    \begin{align*}
        \norm{E(i,:)}_1 &\leq  
         \sum_{\ell = -s_0}^{s_0}\left|\widehat{B}_{i,i+\ell}-B_{i,i+\ell}\right| + \sum_{\ell = s_0+1}^{n-i} |B_{i,i+\ell}| + \sum_{\ell = s_0+1}^{i-1} |B_{i,i-\ell}|\\
        &\leq 2\sum_{\ell = s_0+1}^{n-i} |B_{i,i+\ell}| + 2\sum_{\ell = s_0+1}^{i-1} |B_{i,i-\ell}|\\
        &< 2C\sum_{\ell = s_0+1}^{n-i} \lambda^\ell +  2C\sum_{\ell = s_0+1}^{i-1} \lambda^\ell
        \\ &< 4C \lambda^{s_0+1} \sum_{\ell = 0}^{\infty} \lambda^{\ell} = \frac{4C\lambda}{1-\lambda} \lambda^{s_0}.
    \end{align*} Similarly, the $j$th column of $\mtx{E}$ satisfies
    \begin{equation*}
         \norm{E(:,j)}_1 < \frac{4C\lambda}{1-\lambda} \lambda^{s_0}.
    \end{equation*}
    Therefore
    \begin{equation*}
        \norm{\mtx{E}}_1 = \max_{1\leq j\leq n} \norm{E(:,j)}_1 < \frac{4C\lambda}{1-\lambda} \lambda^{s_0},\qquad 
        \norm{\mtx{E}}_\infty = \max_{1\leq i \leq n} \norm{E(i,:)}_1 <  \frac{4C\lambda}{1-\lambda} \lambda^{s_0}.
    \end{equation*}
    We conclude that 
    \begin{equation*}
         \norm{\widehat{\mtx{B}} - \mtx{B}}_2 \leq \frac{4C\lambda}{1-\lambda} \lambda^{s_0}.
    \end{equation*}
\end{proof}
\begin{remark} \*
\begin{enumerate} 
    \item When $\mtx{B} = f(\mtx{A})$ for a $k$-banded matrix $\mtx{A}$, we can use $s = 2ck+1$ for some integer constant $c>0$. This choice of $s$ is motivated by the fact that if $f$ is analytic then $f(\mtx{A})$ can be well approximated by a low degree polynomial $p(\mtx{A})$ which has bandwidth deg$(p)k$.
    It is also useful to note the dependence of the decay factor $\lambda$ on the bandwidth. If the entries of a $1$-banded matrix decay away from the main diagonal at rate $\lambda$, then for a $k$-banded matrix with comparable spectral properties, the decay rate becomes $\lambda^{1/k}$ \cite{DMS1984}. This behaviour aligns with the graph distance $d(i,j)$ for general sparse matrices discussed in Section~\ref{sec:sparse}, and the choice $s = 2ck+1$.
    \item In Theorem \ref{thm:band}, if $\lambda \leq 0.9$ then for $\epsilon$ accuracy it suffices to take
    \begin{equation*}
        s_0\geq \frac{\log{(36C)}-\log{\epsilon}}{\log{(\lambda^{-1})}}.
    \end{equation*}
    \item
    In many cases, the constants $C,\lambda$ and $s_0$ are either unknown, implicit or too expensive to compute in applications. For matrix functions, Theorem \ref{thm:band} relies on the off-diagonal decay bound in \cite{Benzi2007}, which contains constants that are implicit as they depend on the function $f$ and the spectral property of the matrix $\mtx{A}$. Therefore, without decay bounds that include explicit constants, Theorem \ref{thm:band} cannot serve as a priori estimate for Algorithm \ref{alg:band}. However, the analysis in Theorem \ref{thm:band} can still serve as a heuristic for when Algorithm \ref{alg:band} will be useful.
\end{enumerate}
\end{remark}

\subsubsection{Different decay rates} \label{sec:difdecayrate}
A similar analysis can be performed if the decay rate is different on each side of the diagonal, for example, if the matrix satisfies \eqref{eq:ubandbound} and \eqref{eq:lbandbound}. The argument is the same with a different definition for $t_{ij}$, which is
\begin{equation} \label{eq:diffdecayt}
    t_{ij} = \argmin\limits_{\substack{1\leq j+st \leq n \\ t\in \Z}}\left|j+st-i-\frac{s-1}{2}\frac{\log(\lambda_1^{-1})-\log(\lambda_2^{-1})}{\log(\lambda_1^{-1})+\log(\lambda_2^{-1})}\right|
\end{equation} where $\lambda_1$ is the decay rate above the diagonal and $\lambda_2$ is that below the diagonal. This expression for $t_{ij}$ accounts for the different decay rates and coincides with our analysis in the previous section when $\lambda_1 = \lambda_2$. The algorithm will be the same as \RB{} with $t_{ij}$ defined as \eqref{eq:diffdecayt}. For $s = \flr{c/\log(\lambda_1^{-1})}+\flr{c/\log(\lambda_2^{-1})}+1$ for some constant $c>0$, the algorithm will output $\widehat{\mtx{B}}$ with upper bandwidth $\flr{c/\log(\lambda_1^{-1})}$ and lower bandwidth $\flr{c/\log(\lambda_2^{-1})}$ and $\widehat{\mtx{B}}$ will satisfy
\begin{equation*}
    \norm{\widehat{\mtx{B}}-\mtx{B}}_2 \lesssim e^{-c}
\end{equation*} which can easily be shown by following the proof of Theorem \ref{thm:band}.

\subsection{Extensions to Kronecker Sum}
Some of the results for banded matrices and matrices with exponential off-diagonal decay can be naturally extended to Kronecker sums, which arises, for example in the finite difference discretization of the two-dimensional Laplace operator \cite{LeVequeBook}. In particular, decay bounds for functions of Hermitian matrices with Kronecker structure have been studied in \cite{BenziSimoncini2015}. Let $\mtx{A}^{(1)}, \mtx{A}^{(2)} \in \R^{n\times n}$. The Kronecker sum $\mathcal{A} \in \R^{n^2\times n^2}$ of $\mtx{A}^{(1)}$ and $\mtx{A}^{(2)}$ is defined by
\begin{equation*}
    \mathcal{A} = \mtx{A}^{(1)} \oplus \mtx{A}^{(2)} = \mtx{A}^{(1)} \otimes \mtx{I}_n + \mtx{I}_n \otimes \mtx{A}^{(2)}
\end{equation*} where $\otimes$ is the Kronecker product, defined by 
\begin{equation*}
    \mtx{A}\otimes \mtx{B} = \begin{bmatrix}
        A_{11} \mtx{B} & A_{12} \mtx{B} & ... & A_{1n}\mtx{B} \\
        A_{21} \mtx{B} & A_{22} \mtx{B} & ... & A_{2n}\mtx{B} \\
        \vdots & \vdots & \ddots & \vdots \\
        A_{n1} \mtx{B} & A_{n2} \mtx{B} & ... & A_{nn}\mtx{B}
    \end{bmatrix} \in \R^{n^2\times n^2}.
\end{equation*} 

Suppose $\mtx{A}^{(1)}$ and $\mtx{A}^{(2)}$ are both banded matrices. Then $\mtx{I}_n \otimes \mtx{A}^{(2)} \in \R^{n^2\times n^2}$ is a banded matrix with the same bandwidth as $\mtx{A}^{(2)}$ and $\mtx{A}^{(1)} \otimes \mtx{I}_n \in \R^{n^2 \times n^2}$ is permutation equivalent to a banded matrix with the same bandwidth as $\mtx{A}^{(1)}$ because $\mtx{P}(\mtx{A}^{(1)} \otimes \mtx{I}_n)\mtx{P}^T = \mtx{I}_n \otimes \mtx{A}^{(1)}$ where $\mtx{P}$ is the perfect shuffle matrix \cite{VanLoan2000}. With this observation, we can use the recovery technique used for banded matrices in Section \ref{sec:bandrecovery} to recover $\mathcal{A}$ for banded matrices $\mtx{A}^{(1)}$ and $\mtx{A}^{(2)}$. Suppose that $\mtx{A}^{(1)}$ and $\mtx{A}^{(2)}$ have bandwidth $k^{(1)}$ and $k^{(2)}$ respectively. Notice that to recover $\mathcal{A}$, we need to recover the off-diagonal entries of $\mtx{A}^{(1)}$ and $\mtx{A}^{(2)}$ to recover the off-diagonal entries of $\mathcal{A}$ and recover $A_{ii}^{(1)}+A_{jj}^{(2)}$ for all $i,j$ to recover the diagonal entries of $\mathcal{A}$. We use the following $(2+2k^{(1)}+2k^{(2)})$ matrix-vector products
\begin{equation*}
    \left[\mtx{P}^T\begin{bmatrix}
        \mtx{I}_n^{\left(1+2k^{(1)}\right)} \\ \mtx{0}
    \end{bmatrix}, \begin{bmatrix}
        \mtx{I}_n^{\left(1+2k^{(2)}\right)} \\ \mtx{0}
    \end{bmatrix} \right] \in \R^{n^2 \times \left(2+2k^{(1)}+2k^{(2)}\right)}
\end{equation*} to recover $\mathcal{A}$. Since the top left $n\times n$ block of $\mathcal{A}$ equals $A_{11}^{(1)}\mtx{I}_n+\mtx{A}^{(2)}$, which is $k^{(2)}-$banded, the first $n$ rows of the matrix-vector product with $\begin{bmatrix}
        \mtx{I}_n^{\left(1+2k^{(2)}\right)} \\ \mtx{0}
    \end{bmatrix}$ recover $A_{11}^{(1)}\mtx{I}_n+\mtx{A}^{(2)}$ exactly (Proposition \ref{thm:bband}), which in turn recovers all the off-diagonal entries of $\mtx{A}^{(2)}$. Similarly, $\mtx{P}^T\begin{bmatrix}
        \mtx{I}_n^{\left(1+2k^{(1)}\right)}\\ \mtx{0}
    \end{bmatrix}$ recovers $\mtx{A}^{(1)}+A_{11}^{(2)}\mtx{I}_n$ exactly using $\mtx{P}(\mtx{A}^{(1)} \otimes \mtx{I}_n)\mtx{P}^T = \mtx{I}_n \otimes \mtx{A}^{(1)}$. Therefore we get all the off-diagonal entries of $\mathcal{A}$ from recovering $A_{11}^{(1)}\mtx{I}_n+\mtx{A}^{(2)}$ and $\mtx{A}^{(1)}+A_{11}^{(2)}\mtx{I}_n$. The diagonal entries of $\mathcal{A}$ can be recovered from the diagonal entries of $A_{11}^{(1)}\mtx{I}_n+\mtx{A}^{(2)}$ and $\mtx{A}^{(1)}+A_{11}^{(2)}\mtx{I}_n$ by noting
    \begin{equation*}
        A_{ii}^{(1)}+A_{jj}^{(2)} = (A_{ii}^{(1)} + A_{11}^{(2)}) + (A_{11}^{(1)} + A_{jj}^{(2)}) - (A_{11}^{(1)} + A_{11}^{(2)})
    \end{equation*} for all $i,j$.

For matrices with rapid off-diagonal decay, we can extend the above idea in certain cases. Focusing on the matrix exponential, we have a special relation {\cite[\S 10]{matfunchigham}}
\begin{equation*}
    \exp(\mtx{A}^{(1)} \oplus \mtx{A}^{(2)}) = \exp(\mtx{A}^{(1)}) \otimes \exp(\mtx{A}^{(2)}).
\end{equation*} When $\mtx{A}^{(1)}$ and $\mtx{A}^{(2)}$ are both banded matrices we can use the same procedure above by applying matrix-vector products with
\begin{equation*}
    \left[\mtx{P}^T\begin{bmatrix}
        \mtx{I}_n^{(s_1)} \\ \mtx{0}
    \end{bmatrix}, \begin{bmatrix}
        \mtx{I}_n^{(s_2)} \\ \mtx{0}
    \end{bmatrix} \right] \in \R^{n^2 \times (s_1+s_2)}.
\end{equation*} for sufficiently large $s_1$ and $s_2$ that captures the dominant entries of $\exp(\mtx{A}^{(1)})$ and $\exp(\mtx{A}^{(2)})$ near the main diagonal respectively. This extension is natural and analogous to the extension from recovering banded matrices to recovering matrices with exponential off-diagonal decay earlier this section.

The strategy for recovering banded matrices and approximating matrices with rapid off-diagonal decay can also be used as laid out above in a similar setting involving Kronecker product or sum. Examples include $f(\mtx{I} \otimes \mtx{A}) = \mtx{I} \otimes f(\mtx{A})$ and the formulae for the cosine and the sine of the Kronecker sum between two matrices {\cite[\S 12]{matfunchigham}}.


\section{Numerical Illustrations} \label{sec:numill}
Here we illustrate \RB{} and \RS{} using numerical experiments. In all experiments, we use $20$ iterations of a polynomial Krylov method and $50$ sample points for the contour integration from method 2 of \cite{matveccontour} to evaluate the matrix-vector products with $f(\mtx{A})$. The contour integration requires the knowledge of the largest and smallest eigenvalues of $\mtx{A}$. We assume that these eigenvalues are known a priori. In all of the plots, the $x$-axis denotes the number of matrix-vector products $s$, i.e. $\mtx{x}\mapsto f(\mtx{A})\mtx{x}$. Note that this refers to the matrix-vector products with $f(\mtx{A})$, not with $\mtx{A}$. If computing  $\vct{x}\mapsto f(\mtx{A})\vct{x}$ requires $m$ matrix-vector products with $\mtx{A}$, then approximating $f(\mtx{A})$ in our experiments requires a total of $ms$ matrix-vector products with $\mtx{A}$. The $y$-axis denotes the relative error in the $2$-norm with respect to $f(\mtx{A})$ obtained by dense arithmetic in MATLAB, i.e., \texttt{expm}, \texttt{sqrtm} and \texttt{logm}. Here, we plot the relative error $\|f(\mtx{A})-\widehat{f(\mtx{A}})\|_2/\|f(\mtx{A})\|_2$ instead of the absolute error $\|f(\mtx{A})-\widehat{f(\mtx{A}})\|_2$ to provide a clearer, scale-invariant view of our numerical results. For \RS, the sparsity parameter $k$ is chosen to be $8$ times smaller than the number of matrix-vector products, i.e., $s = 8k$ and the sensing operator $\mtx{Y}$ is Gaussian. The experiments were conducted in MATLAB version 2021a using double precision arithmetic.

\subsection{Synthetic Examples} \label{subsec:num_syneg}
We illustrate synthetic examples for \RB{} and \RS{} using the functions $e^x, \sqrt{1+x}$ and $\log(1+x)$. We consider a synthetic banded matrix and a synthetic sparse matrix given below.
\begin{itemize}
    \item \emph{Banded case:} symmetric $2$-banded matrix $\mtx{A}_B \in \R^{1024\times 1024}$ with its nonzero entries drawn i.i.d. from $\mathcal{N}(0,1)$. $\mtx{A}_B$ has been rescaled to $\norm{\mtx{A}_B}_2 = 0.5$.
    \item \emph{General sparse case:} symmetric sparse matrix $\mtx{A}_S\in \R^{1024\times 1024}$ created using the MATLAB command \texttt{sprandsym} with density $1/1024$. $\mtx{A}_S$ has been rescaled to $\norm{\mtx{A}_S}_2 = 0.5$.
\end{itemize} Symmetry was enforced to ensure that the eigenvalues are real, and $\mtx{A}_B$ and $\mtx{A}_S$ have been rescaled to ensure that their eigenvalues are sufficiently away from $-1$, which is where $\sqrt{1+x}$ and $\log(1+x)$ have singularities. This was done to ensure that the matrix-vector products $f(\mtx{A})\mtx{b}$ can be computed with high accuracy with 20 Krylov steps.

Figure \ref{fig:syneg} shows the accuracy as well as how well the error estimates $\delta_{RE}^{(S)}$ and $\delta_{RE}^{(B)}$ describe the behaviour of the two algorithms. For the banded case in Figure \ref{fig2a}, we see that the error decays exponentially until the error is dominated by either the machine precision or the error from computing the matrix-vector products. This exponential decay rate is consistent with Theorem \ref{thm:band}. The dotted lines indicate the error estimate $\delta_{RE}^{(B)}$. We observe that $\delta_{RE}^{(B)}$, which makes $5$ extra Gaussian measurements, estimate the error well. Therefore, we recommend computing $\delta_{RE}^{(B)}$ in practice, especially in cases when either the bandwidth of the underlying matrix or the matrix function is unknown. 

For the general sparse case in Figure \ref{fig2b}, the error decays quickly with the number of nonzero entries recovered in each row. The sparsity pattern influences the decay rate of \RS, making the behaviour of the decay rate more unpredictable and irregular. This has been shown partially in Theorem \ref{thm:gen} where the bound depends on the sparsity pattern or more specifically the cardinality of the union of the combined support locations of $\mtx{I}_n,\mtx{A},...,\mtx{A}^d$ in each row for $d$ large enough. Although the behavior of \RS{} is irregular, $\delta_{RE}^{(S)}$ serves as an excellent error estimate for how \RS{} performs.

Finally, we note that there are regimes in which the proposed method does not convincingly amortize its computational cost. For example, in Figure \ref{fig2b} for the $1024 \times 1024$ problem, the \RS requires $55$ matrix–vector products with $f(\mtx{A})$—corresponding to $55 \times 19 = 1045$ matrix–vector products with $\mtx{A}$—to achieve only $10^{-3}$ accuracy. In contrast, naive dense methods that explicitly form the matrix and compute a full eigenvalue decomposition achieve full accuracy using fewer matrix–vector products with $\mtx{A}$. This highlights that the advantages of the proposed approach are most pronounced in large-scale settings where forming the matrix is infeasible, rather than in moderate-size problems where dense methods remain competitive. See Section~\ref{subsec:scaletest} for experiments at larger scale.

\begin{figure}[!ht]
\subfloat[]{\label{fig2a}\includegraphics[scale = 0.45]{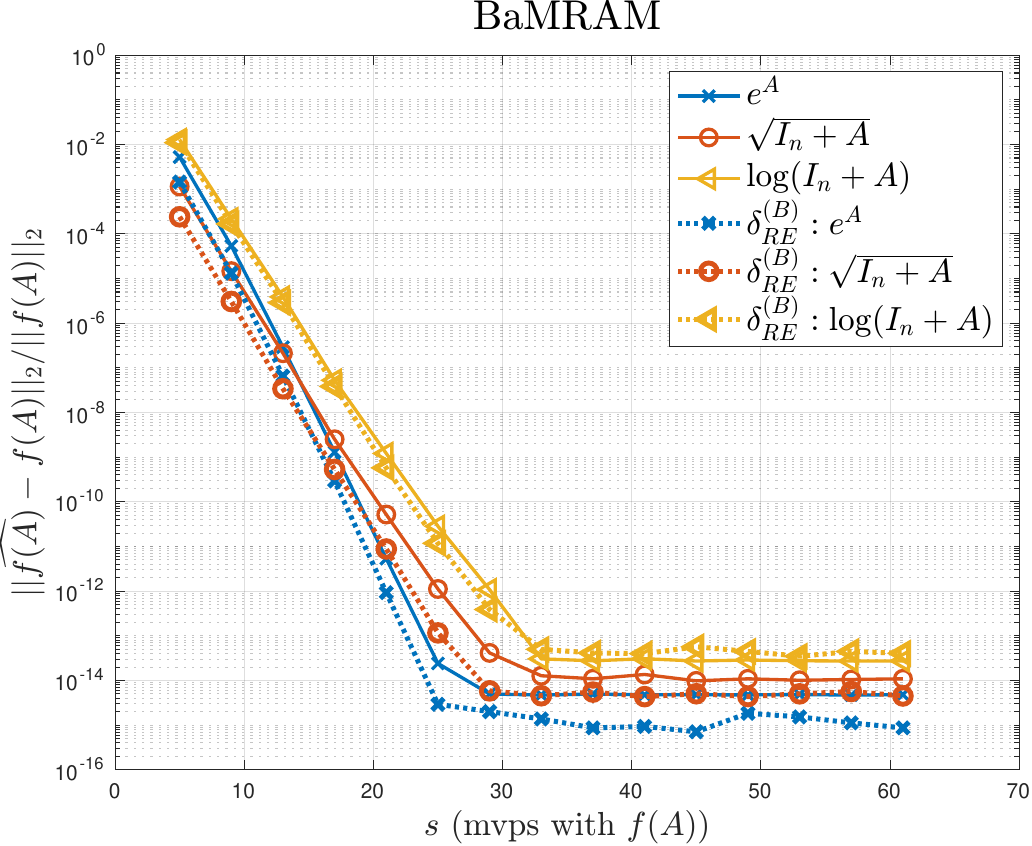}} \hfill
\subfloat[]{\label{fig2b}\includegraphics[scale = 0.45]{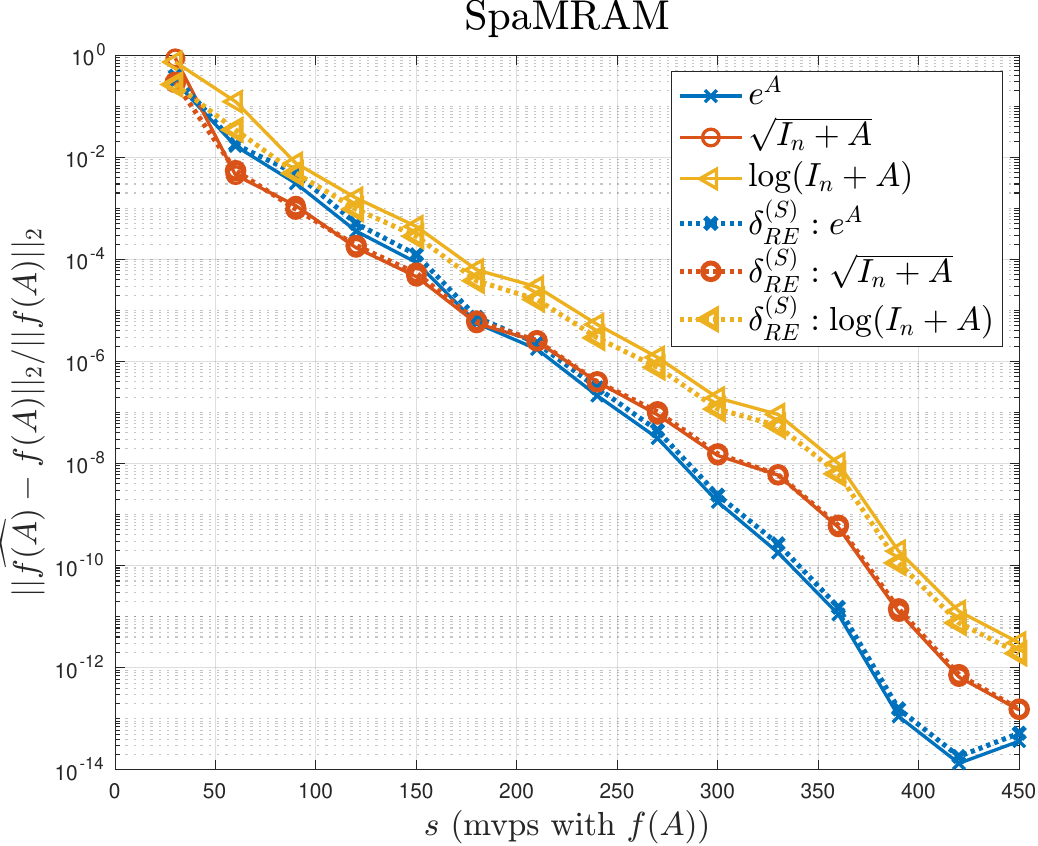}}
\centering 
\caption{The accuracy of \RB{} and \RS{} using two $1024\times 1024$ synthetic symmetric matrices: a random $2$-banded matrix (left) and a random sparse matrix with nonzero density $1/1024$ (right). The dotted lines are the error estimates $\delta_{RE}^{(S)}$ and $\delta_{RE}^{(B)}$ discussed in Sections \ref{subsec:posterrspamram} and \ref{subsec:errestband} respectively. The error estimates can be used to determine the behaviour of \RS{} and \RB{}.}
\label{fig:syneg}
\end{figure}

\subsection{Toy Examples}
In this section, we manufacture a matrix $\mtx{A}$ such that $f(\mtx{A})$ has a sparse structure and use our algorithms to approximate $f(\mtx{A})$ using matrix-vector products with $\mtx{A}$ only. We explore two functions, namely $\sqrt{\mtx{A}}$ and $\log{(\mtx{A})}$. For the square root function, we set $\sqrt{\mtx{A}}$ equal to a sparse positive definite matrix and use our algorithms 
to recover $\sqrt{\mtx{A}}$ via matrix-vector products with $\mtx{A}=\left(\sqrt{\mtx{A}}\right)^2$ (which here is not very sparse). We follow a similar procedure for the matrix logarithm using the formula $\mbox{exp}(\log(\mtx{A}))=\mtx{A}$. 
The aim of this experiment is to show that if $f(\mtx{A})$ is sparse then we can efficiently recover $f(\mtx{A})$ to high accuracy using our algorithms. We choose a sparse positive definite matrix and a banded positive definite matrix from the SuiteSparse Matrix Collection \cite{FloridaDataset2011}.
\begin{itemize}
    \item $\mathtt{gr\_ 30 \_ 30}$: $900\times 900$ $31$-banded symmetric positive definite matrix with $7744$ nonzero entries. This matrix has eigenvalues with $\lambda(A)\in [0.06,11.96]$.
    \item $\mathtt{Trefethen\_ 700}$: $700\times 700$ sparse symmetric positive definite matrix with $12654$ nonzero entries and at most $19$ nonzero entries per row and column. This matrix has eigenvalues with $\lambda(A)\in [1.12,5280]$.
\end{itemize}

In Figure \ref{fig:toyeg}, after an initial period of low accuracy--caused by the unrecovered entries of both matrices being $\bigO(1)$--we obtain high accuracy as the number of nonzero entries recovered per row increases sufficiently for $f(\mtx{A})$, which is either sparse or banded, to be recovered exactly.
In Figure \ref{fig3a}, the matrix $\mathtt{gr\_30\_30}$ is recovered to high accuracy using about $66$ matrix-vector products, which is consistent with the fact that $\mathtt{gr\_30\_30}$ is $31$-banded, and hence has at most $63$ nonzero entries per row and column. This corresponds to a total of $66\times 19 = 1254$ matrix-vector products with $\mtx{A}$.
In Figure \ref{fig3b}, the matrix $\mathtt{Trefethen\_700}$ has been recovered to high accuracy using about $180$ matrix-vector products. This is fairly consistent with the fact that with $s= 180$ matrix-vector products, we recover about $22$ nonzero entries in each row and $\mathtt{Trefethen\_700}$ has at most $19$ nonzero entries per row. This corresponds to $180\times 19 = 3420$ matrix-vector products with $\mtx{A}$.

\begin{figure}[!ht]
\subfloat[]{\label{fig3a}\includegraphics[scale = 0.45]{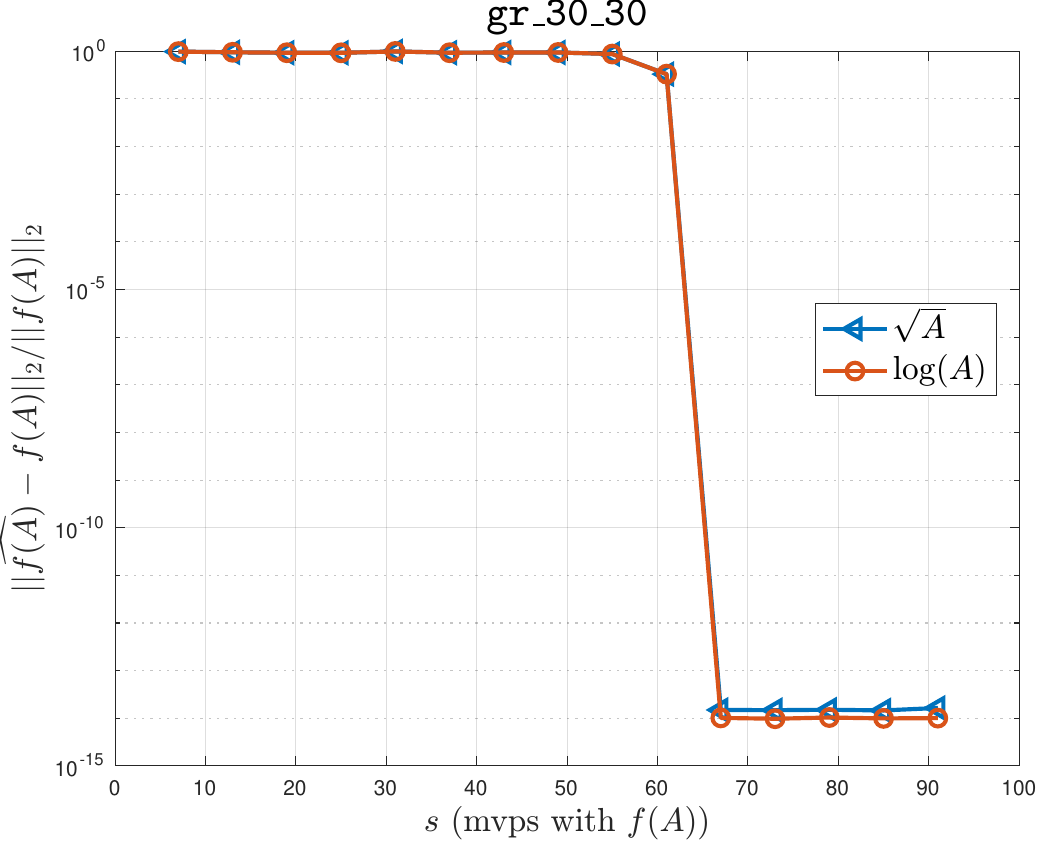}} \hfill
\subfloat[]{\label{fig3b}\includegraphics[scale = 0.45]{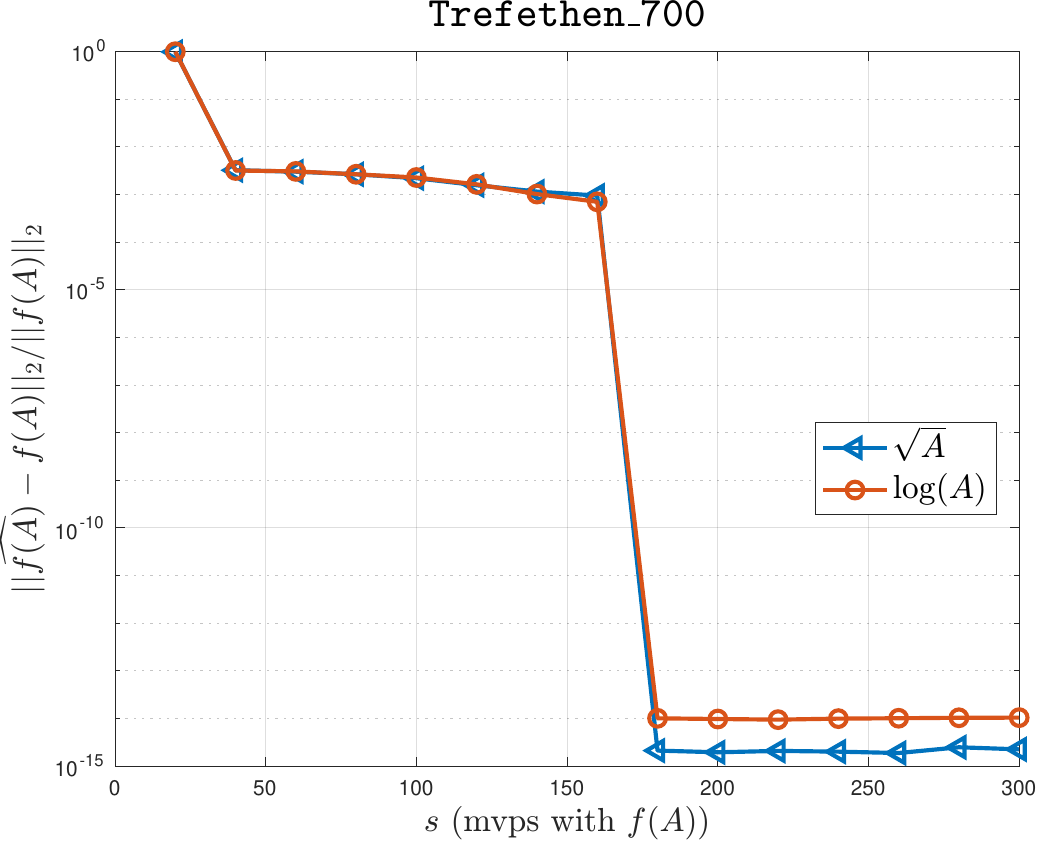}} 
\centering 
\caption{The accuracy of \RB{} and \RS{} for functions of matrices when $f(\mtx{A})$ itself is banded or sparse.}
\label{fig:toyeg}
\end{figure}

\subsection{Sparse dataset examples}
We now illustrate \RB{} and \RS{} using the SuiteSparse Matrix Collection \cite{FloridaDataset2011}, using the four matrices below. The first three matrices are sparse and the last matrix is banded.
\begin{itemize}
    \item $\mathtt{west1505}$: $1505\times 1505$ sparse matrix with $5414$ nonzero entries and at most $12$ nonzero entries per row and at most $27$ nonzero entries per column. This matrix has complex eigenvalues with $|\lambda(A)|\in \left[1.41\cdot 10^{-4},2.29\cdot 10^4\right]$.
    \item $\mathtt{CSphd}$: $1882\times 1882$ sparse matrix with $1740$ nonzero entries and at most $45$ nonzero entries per row and at most $3$ nonzero entries per column. This matrix has all of its eigenvalues equal to zero.
    \item $\mathtt{tols2000}$: $2000\times 2000$ sparse matrix with $5184$ nonzero entries and at most $90$ nonzero entries per row and at most $22$ nonzero entries per column. This matrix has complex eigenvalues with $|\lambda(A)|\in \left[11.79,2442\right]$.
    \item $\mathtt{mhd3200b}$: $3200\times 3200$ $18$-banded symmetric positive definite matrix with $18316$ nonzero entries. This matrix has eigenvalues with $\lambda(A)\in \left[1.37\cdot 10^{-13},2.2\right]$.
\end{itemize}

\begin{figure}[!ht]
\subfloat[$\mathtt{west1505}$]{\label{eig_west1505}\includegraphics[scale = 0.55]{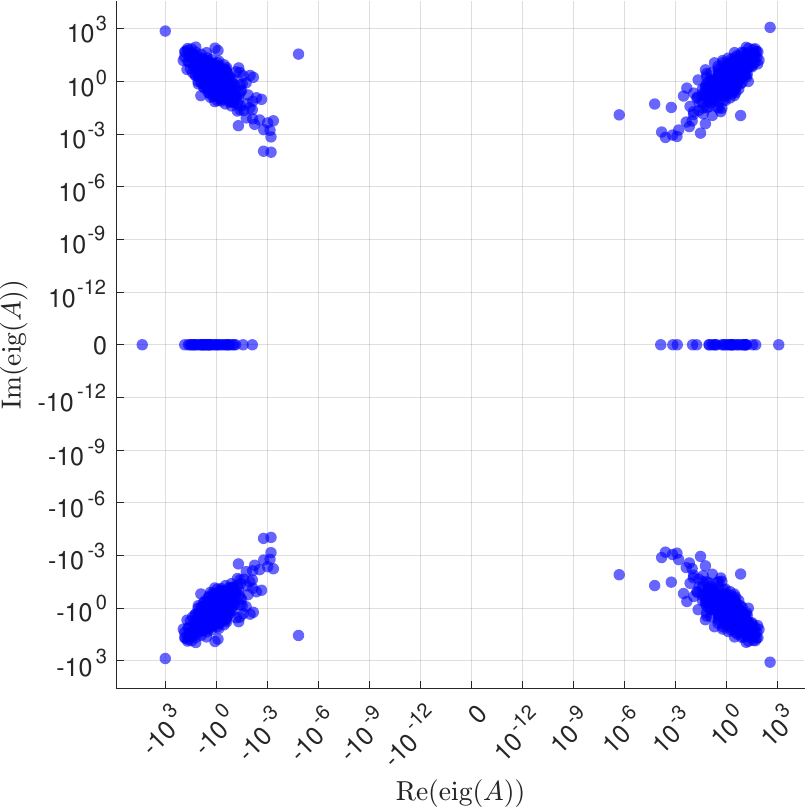}} \hfill
\subfloat[$\mathtt{tols2000}$]{\label{eig_tols2000}\includegraphics[scale = 0.55]{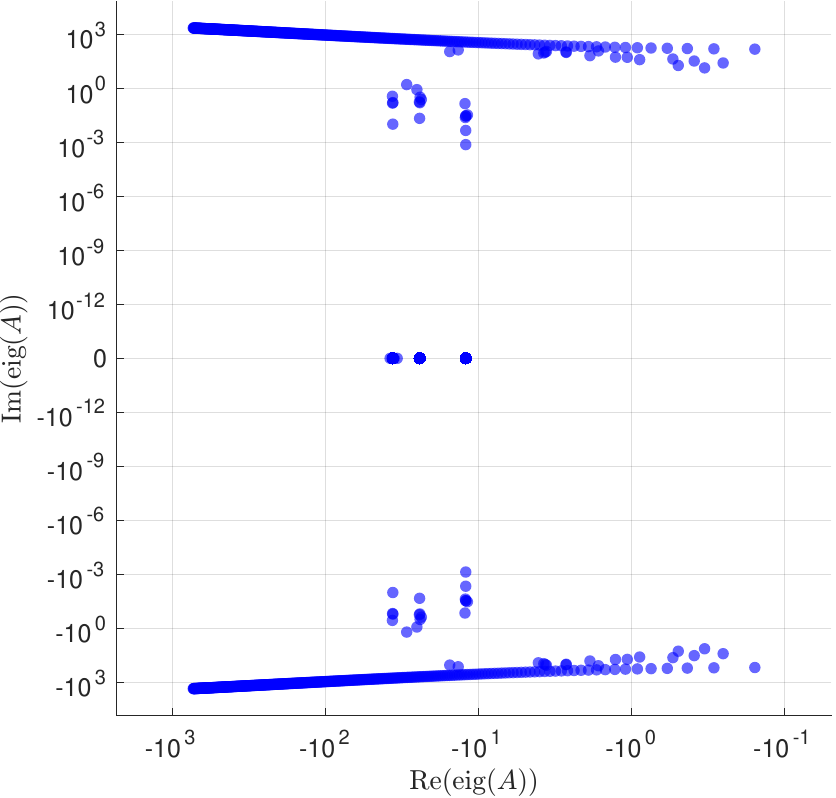}}
\centering 
\caption{Eigenvalue spectrum of $\mathtt{west1505}$ and $\mathtt{tols2000}$}
\label{fig:eigplot}
\end{figure}

The eigenvalue spectrum of $\mathtt{west1505}$ and $\mathtt{tols2000}$ are depicted in Figure \ref{fig:eigplot}. For the sparse matrices $\mathtt{west1505}$, $\mathtt{CSphd}$ and $\mathtt{tols2000}$, we use \RS{} to recover their matrix exponential. For the banded matrix $\mathtt{mhd3200b}$, we use \RB{} to recover its matrix exponential, matrix square root and matrix logarithm. The results are illustrated in Figure \ref{fig:dataseteg}.

\begin{figure}[!ht]
\subfloat[]{\label{fig4a}\includegraphics[scale = 0.45]{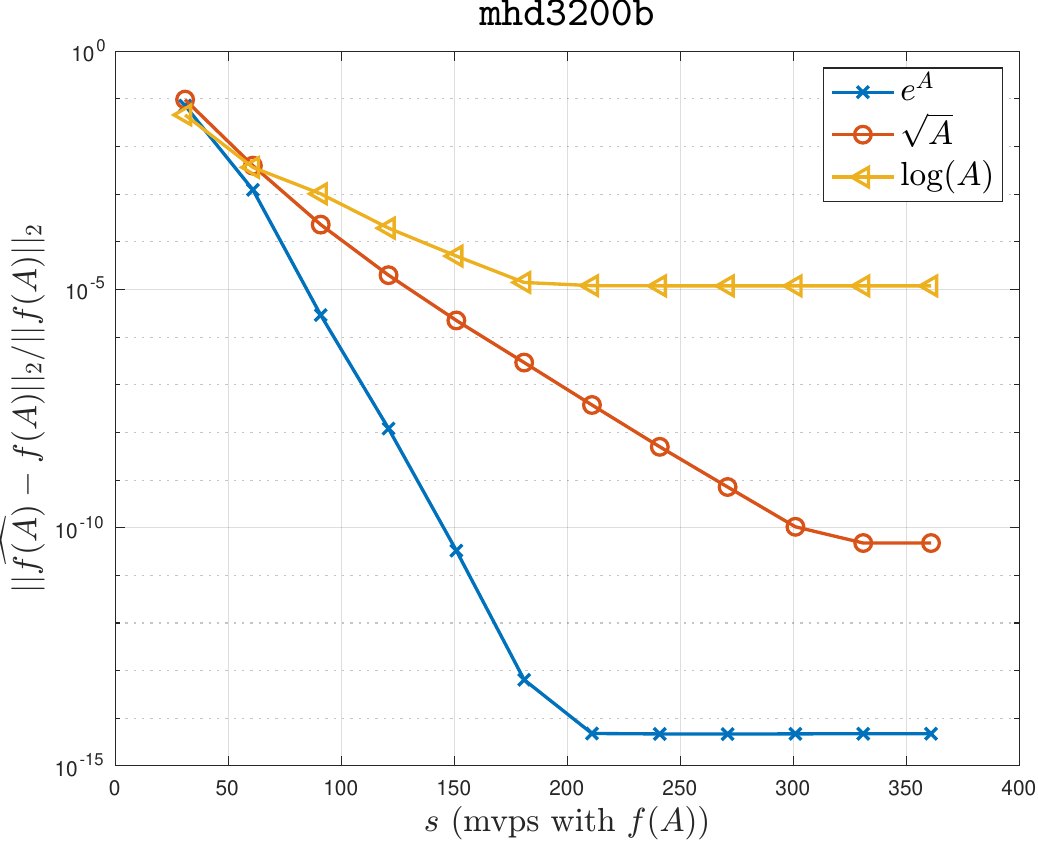}} \hfill
\subfloat[]{\label{fig4b}\includegraphics[scale = 0.45]{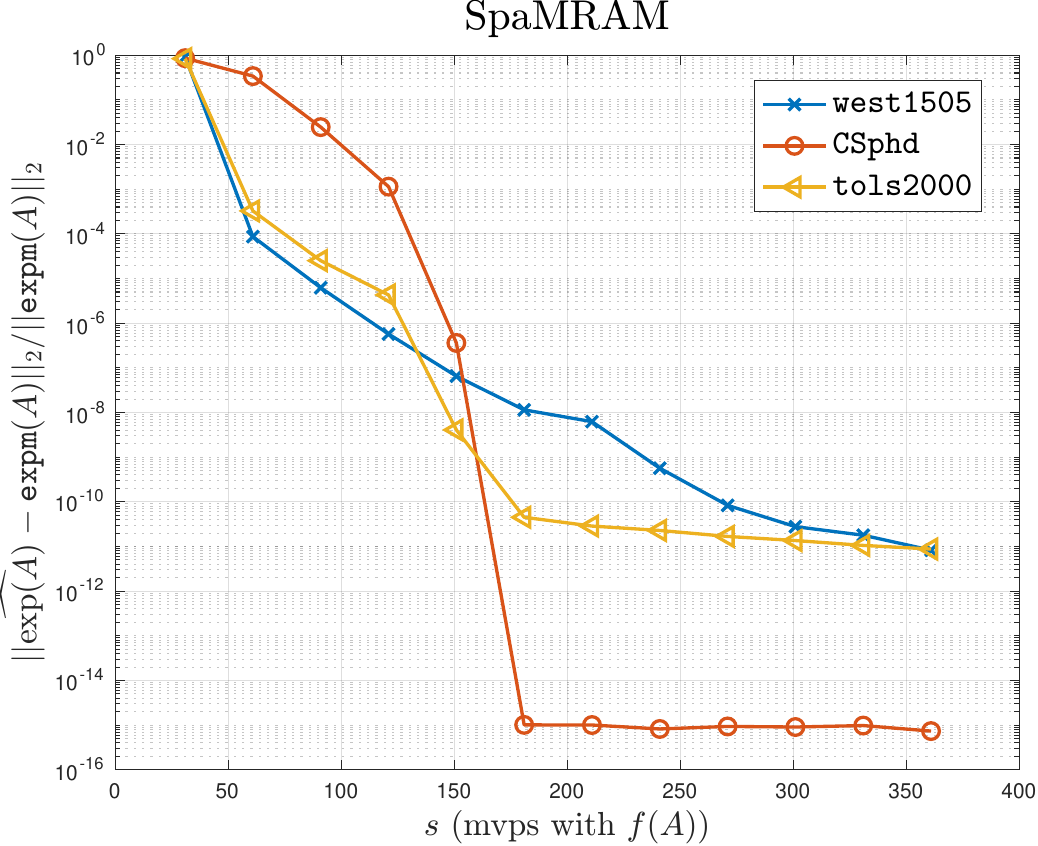}} 
\centering 
\caption{The accuracy of \RB{} 
(left)
and \RS{} (right) using sparse matrices from \cite{FloridaDataset2011}.}
\label{fig:dataseteg}
\end{figure}

In Figure \ref{fig:dataseteg}, we observe that our algorithms are able to approximate $f(\mtx{A})$ with good accuracy. In Figure \ref{fig4a}, we notice that the error decreases exponentially until the error stagnates at some point. The stagnation is due to the error incurred in the computation of matrix-vector products $f(\mtx{A})\vct{b}$, and can be improved by taking more than 20 Krylov steps. 
For banded matrices, we consistently obtain exponential convergence rate whose convergence rate is dictated by the smoothness of $f$. In Figure \ref{fig4b}, we note that the error decreases at a rapid speed until it stagnates due to machine precision for $\mathtt{CSphd}$ or slows down due to the sparsity pattern in the case of $\mathtt{west1505}$ and $\mathtt{tols2000}$. As observed in prior experiments, the behavior is more irregular when the matrix has a more general sparsity pattern.

The approximation to $f(\mtx{A})$ in the above experiments can be highly accurate when the bandwidth of the matrix is small or the matrix has a sparsity pattern that makes $f(\mtx{A})$ approximately sparse, in the sense that only few entries in each row/column are large in magnitude. Our algorithms are particularly useful when $\mtx{A}$ can only be accessed through matrix-vector products and we require an approximation to the entire matrix function. An example is in network analysis where we are given an adjacency matrix $\mtx{A}$ and the off-diagonal entries of $e^{\mtx{A}}$ describe the communicability between two vertices and the diagonal entries of $e^{\mtx{A}}$ describe the subgraph centrality of a vertex \cite{BenziEstradaKlymko2013,EstradaHatano2008,EstradaHigham2010}. The sparse matrix $\mathtt{CSphd}$ in the above experiment is an adjacency matrix for some underlying graph, and with about $180$ matrix-vector products we can approximate its matrix exponential to $10^{-15}$ accuracy. This corresponds to a total of $180\times 19 = 3420$ matrix-vector products with $\mtx{A}$. This means that our algorithm can approximate the communicability between any two vertices and the subgraph centrality of any vertex of $\mathtt{CSphd}$ with $10^{-15}$ accuracy.

\subsection{Scalability test} \label{subsec:scaletest}
This section investigate the asymptotic scalability of \RB{} and \RS{} on matrices larger than those considered in the previous sections. We consider the following clases of synthetic matrices, similar to those in Section~\ref{subsec:num_syneg}:
\begin{itemize}
    \item \emph{Banded case:} $n \times n$ symmetric $2$-banded matrix with its nonzero entries drawn i.i.d. from $\mathcal{N}(0,1)$. The matrix is rescaled so that its $2$-norm equals $0.5$.
    \item \emph{Sparse case:} $n\times n$ symmetric sparse matrix created using the MATLAB command \texttt{sprandsym} with density $\frac{1}{2n}$. This matrix is also rescaled to have $2$-norm equal to $0.5$.
\end{itemize}
We vary the matrix dimension $n$ from $400$ to $25600$ and record the number of matrix-vector products with $\mtx{A}$ required for the methods to meet a prescribed tolerance. The dotted line in the figures corresponds to the number of matrix-vector products with $\mtx{A}$ equal to $n$, which is the number required to recover a general dense matrix. Throughout, we use the matrix exponential, where the action $\vct{x} \mapsto f(\mtx{A})\vct{x}$ is approximated using $10$ Krylov iterations. All other parameters are identical to those used in the previous experiments.

\begin{figure}[!ht]
\subfloat[\RB{}]{\label{fig5a}\includegraphics[scale = 0.45]{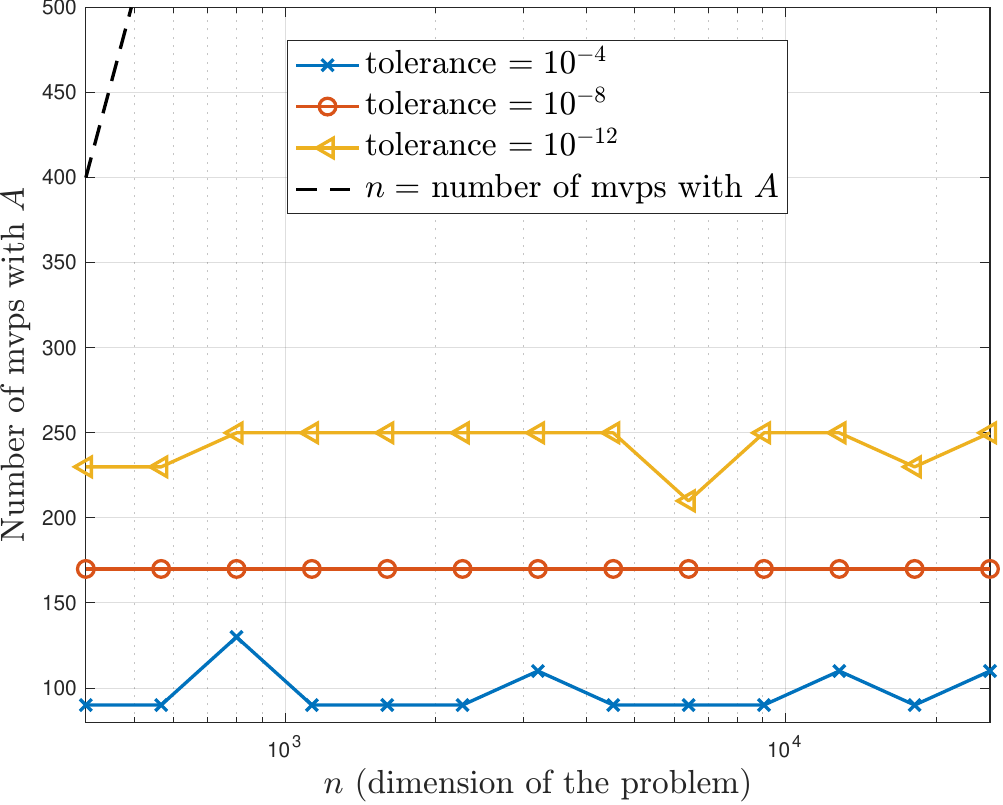}} \hfill
\subfloat[\RS{}]{\label{fig5b}\includegraphics[scale = 0.45]{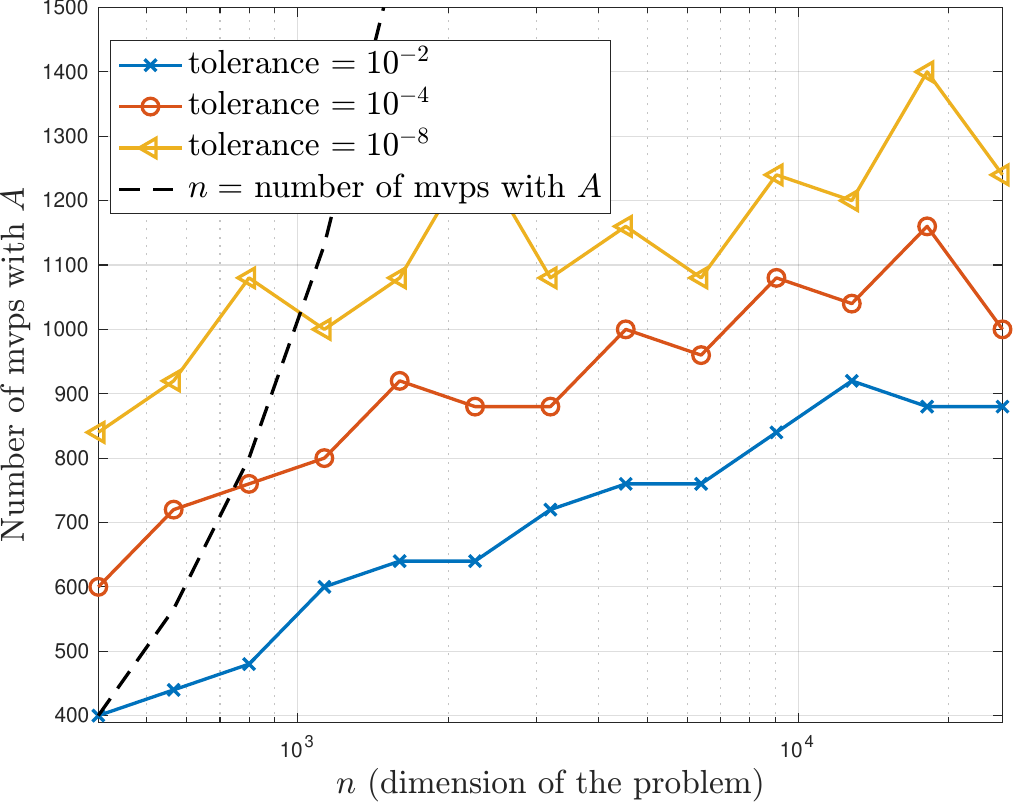}} 
\centering 
\caption{Scalability of \RB{} (left) and \RS{} (right) for matrices with roughly $\bigO(1)$ nonzeros per row and column.}
\label{fig:scaletest}
\end{figure}

The results are shown in Figure~\ref{fig:scaletest}. Both matrix families have $\bigO(1)$ nonzero entries per row and column, which suggests that the corresponding matrix exponential may contain only a small number of dominant entries.

In the banded case (Figure~\ref{fig5a}), we observe that the number of matrix-vector products with $\mtx{A}$ required by \RB{} is essentially independent of the matrix dimension for all tolerances considered. This suggests that, for matrices with bounded bandwidth, \RB{} can effectively recover the dominant entries of $f(\mtx{A})$ using a number of matrix-vector products with $\mtx{A}$ that does not grow with $n$.

For the sparse case, \RS{} exhibits a more irregular behavior as the matrix dimension increases, which can be attributed to the unknown sparsity pattern of the original matrix $\mtx{A}$. Nevertheless, the number of matrix-vector products grows slowly with $n$ and remains below the dotted line for large $n$ and moderate tolerances. Although the curves exhibit mild fluctuations, the overall scaling is sublinear for this test matrix, indicating that the \RS{} effectively exploits the localized structure of the underlying matrix function. In both cases, tighter tolerances lead to an increased number of matrix-vector products, while still maintaining a favorable complexity compared to recovering a general dense matrix.

\section{Discussion and extensions} \label{sec:disc}
In this paper, we presented two algorithms for approximating functions of sparse matrices and matrices with similar structures using matrix-vector products only. We exploited the decay bound from \cite{Benzi2007} to recover the entries of $f(\mtx{A})$ that are large in magnitude. The task of approximating functions of matrices is not the only application for \RS{} and \RB{}. These two algorithms can be used more generally in applications where we are able to perform matrix-vector products with the matrix in question and the matrix we want to approximate satisfies a decay bound similar to \eqref{eq:sparsedecay} or \eqref{eq:symdecayrate}. An example is the decay bound for spectral projectors \cite{BenziBoitoRazouk2013}, computed, for example by a discretized contour integral. Therefore the algorithms can be extended naturally to other similar applications.

In many applications, we also want to approximate quantities related to $f(\mtx{A})$ such as the trace of $f(\mtx{A})$ and the Schatten $p$-norm of $f(\mtx{A})$, which for the trace of functions of sparse symmetric matrices has been studied in \cite{FrommerRinelliSchweitzer2023}. Since \RS{} and \RB{} provide an approximation to $f(\mtx{A})$, we can use $\widehat{f(\mtx{A})}$ to estimate the quantity of interest. For example, the trace of a matrix function is an important quantity appearing in network analysis through the Estrada index \cite{EstradaHigham2010} and statistical learning through the log-determinant estimation \cite{CortinovisKressner2021}, among others. Therefore the algorithms can also be used to estimate quantities related to matrix functions.

 \RS{} uses compressed sensing to recover the matrix function row by row. This algorithm is simple, however we may have access to extra information such as the sparsity pattern of the matrix. This information can potentially be used to speed up the algorithm and increase the accuracy of the algorithm. There are techniques in compressed sensing such as variance density sampling \cite{KrahmerWard2014}, which improves signal reconstruction using prior information. The improvement of \RS{} and the use of \RS{} and \RB{} in different applications are left for future work. 

\section*{Acknowledgements} The authors are grateful to Jared Tanner for his helpful discussions and for pointing out useful literature in compressed sensing. 


\bibliographystyle{abbrv}

\bibliography{cas-refs}






\end{document}